\documentclass[a4paper,11pt]{amsart}
\usepackage[active]{srcltx}
\usepackage{graphicx}
\usepackage{amssymb}
\usepackage[english]{babel}
\usepackage{stmaryrd}
\usepackage{enumerate}
\usepackage[T1]{fontenc}
\usepackage[latin1]{inputenc}
\usepackage[curve,color,line]{xypic}
\usepackage{xcolor}
\usepackage{hyperref}

\renewcommand{\ge}{\geqslant}
\renewcommand{\le}{\leqslant}
\let\op=\llbracket
\let\cl=\rrbracket
\def\pv#1{\ensuremath{\mathsf{#1}}}
\def\Om#1#2{\ensuremath{\overline{\Omega}_{#1}{\pv{#2}}}}
\let\cal=\mathcal
\def\Cl#1{\ensuremath{\cal#1}}

\newcommand\malcev{\mathbin{\bigcirc\kern-8.5pt%
\raise1pt\hbox{\footnotesize$m$}\kern1pt}}

\def\variable{\underline{\ \ }}

\newtheorem{Thm}{Theorem}[section]
\newtheorem{Prop}[Thm]{Proposition}
\newtheorem{Lemma}[Thm]{Lemma}
\newtheorem{Cor}[Thm]{Corollary}
\begin{document}

\title[Representations, irreducibility, order
primitivity]{Representations of relatively free profinite semigroups,
  irreducibility, and order primitivity}

\author{J. Almeida}%
\address{CMUP, Dep.\ Matem\'atica, Faculdade de Ci\^encias, Universidade do
  Porto, Rua do Campo Alegre 687, 4169-007 Porto, Portugal}
\email{jalmeida@fc.up.pt}

\author{O. Kl\'ima}%
\address{Dept.\ of Mathematics and Statistics, Masaryk University,
  Kotl\'a\v rsk\'a 2, 611 37 Brno, Czech Republic}%
\email{klima@math.muni.cz}

\begin{abstract}
  We establish that, under certain closure assumptions on a
  pseudovariety of semigroups, the corresponding relatively free
  profinite semigroups freely generated by a non-singleton finite set
  act faithfully on their minimum ideals. As applications, we enlarge
  the scope of several previous join irreducibility results for
  pseudovarieties of semigroups, which turn out to be even join
  irreducible in the lattice of pseudovarieties of ordered semigroups,
  so that, in particular, they are not generated by proper
  subpseudovarieties of ordered semigroups. We also prove the stronger
  form of join irreducibility for the Krohn-Rhodes complexity
  pseudovarieties, thereby solving a problem proposed by Rhodes and
  Steinberg.
\end{abstract}


\makeatletter%
\@namedef{subjclassname@2010}{%
  \textup{2010} Mathematics Subject Classification}%
\makeatother


\maketitle

\tableofcontents

\section{Introduction}
\label{sec:intro}

Finite semigroups appear naturally in computer science as transition
semigroups of finite automata, which makes them into algebraic
recognition devices for regular languages. A more direct connection is
obtained by associating with a regular language its syntactic
semigroup, namely the quotient of the free semigroup on the underlying
alphabet in which two words are identified if they appear in the same
contexts with respect to the language. The syntactic semigroup is thus
naturally ordered by comparison of the contexts of words. The
properties of classes of regular languages that are captured by
algebraic properties of their syntactic semigroups, respectively of
their ordered syntactic semigroups, have been characterized in terms
of closure properties with respect to certain natural combinatorial
operators. Such classes of regular languages are known as varieties,
respectively positive varieties, of regular languages. The algebraic
counterparts are the so-called pseudovarieties of semigroups
\cite{Eilenberg:1976}, respectively pseudovarieties of ordered
semigroups \cite{Pin:1995a}, both characterized by natural algebraic
closure properties.

Thus, since the syntactic semigroup can be effectively computed, to
determine whether a given regular language belongs to a given variety
becomes a decision problem on the corresponding pseudovariety, namely
to determine whether a given finite semigroup belongs to it. Natural
operators on varieties of languages correspond to natural operators on
pseudovarieties of semigroups. But, such operators are often defined
in terms of generators, rather than by characteristic properties of
their members. The expression of a pseudovariety in terms of simpler
pseudovarieties involving those operators, besides having structural
significance, sometimes leads to decision procedures for the
membership problem. However, whether the existence of such procedures
may be inferred depends on the operators involved. In fact, the
membership problem for pseudovarieties admitting decompositions in
terms of several operators may be rather difficult
\cite{Henckell&Margolis&Pin&Rhodes:1991}, and even undecidable
\cite{Albert&Baldinger&Rhodes:1992,Auinger&Steinberg:2001b}.

One particularly simple operator on pseudovarieties is the join, in
the lattice of pseudovarieties. The existence of nontrivial join
decompositions, in the strict sense, or more generally of nontrivial
join covers, has been investigated by several authors. Some
pseudovarieties admit non-obvious join decompositions
\cite[Chapter~9]{Almeida:1994a}, whereas some have been shown not to
admit any nontrivial join covers
\cite{Margolis&Sapir&Weil:1995,Rhodes&Steinberg:2004,
  Rhodes&Steinberg:2009qt}. Basically, two approaches have been
devised to handle this problem: the syntactical approach, through
pseudoidentities, which may be used to define pseudovarieties
\cite{Reiterman:1982,Molchanov:1994,Pin&Weil:1996b}; and the
structural approach, through the investigation of special structural
properties of generators of the pseudovarieties, such as the
Kov\'acs-Newman property \cite[Section~7.4]{Rhodes&Steinberg:2009qt}.
Recently, we have improved the results
of~\cite{Margolis&Sapir&Weil:1995} using a variant of the syntactical
approach used in~\cite{Almeida&Klima:2011a}. In the present paper, we
combine the two approaches to obtain results that cover and improve
most of the previous join irreducibility results found in the
literature. We are also able to prove join irreducibility of the
Krohn-Rhodes complexity pseudovarieties, which solves part of
\cite[Problem~43]{Rhodes&Steinberg:2009qt}. Furthermore, our approach
yields yet a finer result: the pseudovarieties in question are in fact
join irreducible in the larger lattice of pseudovarieties of ordered
semigroups.

Pseudoidentities are formal equalities between members of relatively
free profinite semigroups. Relatively free profinite semigroups have a
rich and often mysterious structure. Like any semigroup with a minimum
ideal, they act by left and right multiplication on their minimum
ideals. A key property considered in~\cite{Rhodes&Steinberg:2009qt} in
the finite case, not just in connection with the join irreducibility
question, is that both such representations be faithful. A somewhat
weaker property, which has apparently not been considered before, and
is much easier to establish, is that the action of elements outside
the minimum ideal $K$ can be distinguished, among themselves and in
comparison with those of~$K$, by their action by multiplication on
each side of~$K$. Combined with an additional closure property
involving a certain Rees matrix extension, we show that this is enough
to prove join irreducibility of the pseudovariety in the lattice of
pseudovarieties of ordered semigroups. Alternatively, the 
assumption that the corresponding variety of languages is closed under
concatenation also leads to the same conclusion.

A key technique in this paper is thus to consider the left and right
actions of a profinite semigroup on its minimum ideal, that is, the
natural representation of the semigroup in the translational hull of
the minimum ideal. This combines the discrete and topological cases
considered, respectively
in~\cite[Section~5.5.1]{Rhodes&Steinberg:2009qt}
and~\cite[Chapter~4]{Carruth&Hildebrant&Koch:1986}. If both left and
right components of that natural representation are faithful, then the
minimum ideal is reductive and it follows that its translational hull
is a profinite semigroup.

We further establish the faithfulness of both representations for
relatively free profinite semigroups on several pseudovarieties. On
the other hand, we show that a profinite semigroup for which both
representations are faithful admits no nontrivial closed partial order
compatible with multiplication. An application is that, if, for the
finitely generated free profinite semigroup over a pseudovariety on an
arbitrarily large number of generators, both representations are
faithful, then the pseudovariety is not generated by any proper
subpseudovariety of ordered semigroups. To establish such a property
was in fact the original motivation for the present work. Although
this property is a consequence of join irreducibility in the lattice
of pseudovarieties of ordered semigroups, the result opens up the
potential range of applications, as it requires no closure properties
on the pseudovariety, unlike our results on join irreducibility.

A summary of the main applications of our techniques and related
problems which are left open is given in a table at the end of
Section~\ref{sec:irreducibility}.

\section{Preliminaries}
\label{sec:prelims}

This paper owes much to the book \cite{Rhodes&Steinberg:2009qt}, which
facilitated the access and further improved many key ideas in finite
semigroup theory which were previously dispersed through many research
papers. Another basic reference in the area is~\cite{Almeida:1994a}.
The reader is referred to those books for undefined notions and
notation, as well as general background in the area.

Throughout this paper, (locally) compact spaces are assumed to be
Hausdorff.

\subsection{Some pseudovarieties and operations on them}
\label{sec:pvs-and-operations}

For the reader's convenience, the following is a list of
pseudovarieties of semigroups that play a role in this paper. Each
item in the list is described by a characteristic property of its
elements as well as by a basis or bases of pseudoidentities.
\begin{description}
\item[\pv S] all, $\op x=x\cl$.
\item[\pv I] trivial, $\op x=y\cl$.
\item[\pv{Sl}] semilattices, $\op x^2=x, xy=yx\cl$.
\item[\pv N] nilpotent, $\op x^\omega=0\cl$.
\item[\pv D] definite, $\op xy^\omega=y^\omega\cl$.
\item[$\pv D_n$] definite of degree $n$, $\op xy_1\cdots y_n=y_1\cdots
  y_n\cl$.
\item[\pv K] reverse definite, $\op x^\omega y=x^\omega\cl$.
\item[$\pv K_n$] reverse definite of degree $n$, $\op x_1\cdots x_ny=x_1\cdots
  x_n\cl$.
\item[\pv{LI}] locally trivial, $\op x^\omega yx^\omega=x^\omega\cl$.
\item[\pv{LSl}] local semilattices, $\op x^\omega yx^\omega
  yx^\omega=x^\omega yx^\omega, x^\omega yx^\omega zx^\omega=x^\omega
  zx^\omega yx^\omega\cl$.
\item[\pv{LZ}] left zero, $\op xy=x\cl=\pv K_1$.
\item[\pv{RZ}] right zero, $\op xy=y\cl=\pv D_1$.
\item[\pv{RB}] rectangular bands, $\op x^2=x, xyx=x\cl$.
\item[\pv B] bands, $\op x^2=x\cl$.
\item[\pv A] aperiodic, $\op x^{\omega+1}=x^\omega\cl$.
\item[\pv G] groups, $\op x^\omega=1\cl$.
\item[$\pv{Ab}_n$] Abelian groups of exponent dividing $n$, %
  $\op x^n=1, xy=yx\cl$.
\item[\pv{ReG}] rectangular groups, $\op xy^\omega x^\omega=x\cl$.
\item[\pv{CS}] completely simple, $\op x(yx)^\omega=x\cl$.  
\item[\pv{CR}] completely regular, $\op x^{\omega+1}=x\cl$.
\item[\pv J] \Cl J-trivial, $\op (xy)^\omega=(xy)^\omega x=(yx)^\omega\cl$.
\item[\pv{DA}] regular \Cl D-classes are aperiodic subsemigroups,\newline
 $\op ((xy)^\omega x)^2=(xy)^\omega x\cl %
  =\op(xy)^\omega(yx)^\omega(xy)^\omega=(xy)^\omega, x^{\omega+1}=x^\omega\cl$.
\item[\pv{DO}] regular \Cl D-classes are rectangular groups,\newline
  $\op(xy)^\omega(yx)^\omega(yx)^\omega=(xy)^\omega\cl$.
\item[\pv{DS}] regular \Cl D-classes are subsemigroups,\newline %
  $\op ((xy)^\omega x)^{\omega+1}=(xy)^\omega x\cl %
  =\op((xy)^\omega(yx)^\omega(xy)^\omega)^\omega=(xy)^\omega\cl$.
\end{description}
It is well known that $\pv D=\bigcup_{n\ge1}\pv D_n$ and $\pv
K=\bigcup_{n\ge1}\pv K_n$. For a pseudovariety \pv H of groups,
$\bar{\pv H}$ denotes the pseudovariety of all finite semigroups all
of whose subgroups belong to~\pv H.

A pseudovariety of semigroups \pv V is said to be \emph{monoidal} if
it is generated by its monoids; equivalently, whenever a semigroup $S$
belongs to~\pv V, so does the smallest monoid $S^1$ containing~$S$.
This is the case, for example for the pseudovarieties \pv{DS},
\pv{DO}, \pv{CR}, and for those of the form~$\bar{\pv H}$.
Moreover, the intersection of monoidal pseudovarieties is again a
monoidal pseudovariety.

Recall that a \textrm{relational morphism} of semigroups is a relation
$\mu:S\to T$ whose domain is~$S$ and such that $\mu$ is a subsemigroup
of~$S\times T$. In particular, a homomorphism of semigroups is a
relational morphism. For a given pseudovariety of semigroups
\pv U, a
relational morphism $\mu:S\to T$ is a \emph{\pv
  U-relational morphism} if, for every idempotent $e\in T$, the
subsemigroup $\mu^{-1}(e)=\{s\in S:(s,e)\in\mu\}$ of~$S$ belongs
to~$\pv U$. A \emph{\pv U-homomorphism} is a homomorphism which is
also a \pv U-relational morphism.

The Mal'cev product $\pv U\malcev\pv V$ of the pseudovarieties \pv U
and \pv V may be defined as the pseudovariety generated by the finite
semigroups $S$ for which there is a \pv U-homomorphism $S\to T$ into
some $T\in\pv V$. Equivalently, %
$\pv U\malcev\pv V$ consists of all finite semigroups $S$ for which
there is some \pv U-relational morphism $S\to T$ into some $T\in\pv
V$.

For a semigroup $S$, let $S^I$ be the monoid that is obtained from $S$
by adding a new neutral element, even $S$ already has one. Note that,
if $S=S^1$, then $S^I$ is isomorphic to the subsemigroup
$S\times\{0\}\cup\{(1,1)\}$ of $S\times U_1$, where $U_1=\{0,1\}$ is a
semilattice under the usual product.

Let \pv U and \pv V be monoidal pseudovarieties and suppose that \pv U
contains \pv{Sl}. Suppose that $\mu:S\to T$ is a \pv U-relational
morphism into a semigroup $T\in\pv V$. Let $\nu:S^I\to T^1$ be the
relation given by $\nu=\mu\cup\{(I,1)\}$. Then, $\nu$~is a \pv
U-relational morphism into a semigroup from~\pv V. Hence, $S^I$
belongs to $\pv U\malcev\pv V$, which shows that this pseudovariety is
also monoidal.

It is well known that the Mal'cev product satisfies the following law
\cite[Exercise~2.3.20]{Rhodes&Steinberg:2009qt}:
\begin{equation*}
  \label{eq:Malcev-subassociativity}
  \pv U\malcev(\pv V\malcev\pv W) %
  \subseteq %
  (\pv U\malcev\pv V)\malcev\pv W.
\end{equation*}
In particular, if \pv U~is a fixed point of the operator %
$\pv U\malcev\variable$, then this operator is idempotent and so its
fixed points are precisely the pseudovarieties of the form $\pv
U\malcev\pv V$, where \pv V~is an arbitrary pseudovariety. Moreover,
since %
$\pv U\malcev\bigcap_{i\in I}\pv V_i %
\subseteq %
\bigcap_{i\in I}(\pv U\malcev\pv V_i) %
$, the set of fixed points of the operator $\pv U\malcev\variable$ is
then a complete meet subsemilattice of the lattice of all
pseudovarieties of semigroups.

Examples of pseudovarieties \pv U satisfying the equation %
$\pv U\malcev\pv U=\pv U$ of particular interest in this paper are \pv
A, \pv{DA}, \pv B, \pv D, \pv K, \pv{LI}, \pv{LZ}, \pv{RB}, and
\pv{RZ}, although several others in the above list have the same
property.

We adopt the following definition of semidirect product in the
semigroup setting. Given semigroups $S$ and $T$, and a monoid
homomorphism from $T^1$ into the monoid of endomorphisms of~$S$, the
associated \emph{semidirect product} $S*T$ consists of the set
$S\times T$ with the operation
$(s_1,t_1)(s_2,t_2)=(s_1\,\vphantom{|}^{t_1}s_2,t_1t_2)$, where
$\vphantom{|}^{t_1}s_2$ denotes the image of~$s_2$ under the
endomorphism of~$S$ corresponding to~$t_1$. The \emph{semidirect
  product} of the pseudovarieties of semigroups \pv V and~\pv W is the
pseudovariety generated by all semigroups of the form $S*T$ with
$S\in\pv V$ and $T\in\pv W$. This produces an associative operation on
pseudovarieties of semigroups but the reader is warned that it is not
the definition adopted by some authors.
See~\cite[Example~2.4.24]{Rhodes&Steinberg:2009qt} for a comparison
with the definition adopted in that book. With our definition, the
semidirect product of monoidal pseudovarieties is monoidal
\cite[Exercise~10.2.4]{Almeida:1994a}.

\subsection{The de Bruijn encoding}
\label{sec:deBruijn}

For a pseudovariety of semigroups \pv V and a finite set $A$, \Om AV
denotes the pro-\pv V semigroup freely generated by~$A$. Elements
of~\Om AV will in general be called \emph{pseudowords}.

Let \pv V be a pseudovariety containing $\pv D_n$. For a pseudoword
$w\in\Om AV$, denote by $\mathrm{t}_n(w)$ the longest suffix of~$w$ of
length $|w|$ at most~$n$. By looking at the natural projection $\Om
AV\to\Om AD_n$, one sees immediately that there is only one such
suffix, which justifies the notation. Dually, under the hypothesis
that \pv V contains $\pv K_n$, $\mathrm{i}_n(w)$ denotes the longest
prefix of~$w$ of length at most~$n$.

There is a convenient solution of the pseudoidentity problem for
pseudovarieties of the form~$\pv V*\pv D_n$
\cite[Section~10.6]{Almeida:1994a}, which we proceed to describe.
Denote by $A_k$ the set of all words of length~$k$ in~$A^+$ and by
$A_{\le k}$ all words of length at most~$k$. There is a unique
continuous mapping $\Phi_n:\Om AS\to(\Om{A_{n+1}}S)^1$ with the
following properties:
\begin{enumerate}[(a)]
\item\label{item:Phi-a} $\Phi_n(w)=1$ for every $w\in A_{\le n}$;
\item\label{item:Phi-b} $\Phi_n(w)=w$ for $w\in A_{n+1}$;
\item\label{item:Phi-c} $\Phi_n(uv) %
  =\Phi_n(u)\Phi_n(\mathrm{t}_n(u)\,v) %
  =\Phi_n(u\,\mathrm{i}_n(v))\Phi_n(v)$ %
  for all $u,v\in\Om AS$.
\end{enumerate}
For a word $w\in A^+$, $\Phi_n(w)$ is the word obtained by reading,
from left to right, the successive factors of~$w$ of length~$n+1$. In
case $n=0$, this is just the identity mapping. 
For $n>0$, the word
$\Phi_n(w)$ can thus be thought of as describing a path in the de
Bruijn graph of~$A$ of order~$n$, that is an element of the free
category on this graph. In general, for $n>0$, the pseudoword
$\Phi_n(w)$ can be viewed as an element of the free profinite category
on the same graph. Note that, for $n>0$, the mapping $\Phi_n$ is not a
homomorphism.

\begin{Thm}[{\cite[Theorem~10.6.12]{Almeida:1994a}}]
  \label{t:word-problem-for-V*Dn}
  Let\/ \pv V be a pseudovariety that contains some nontrivial monoid
  and let $n>0$. A pseudoidentity $u=v$ holds in the
  pseudovariety~$\pv V*\pv D_n$ if and only if\/
  $\mathrm{i}_n(u)=\mathrm{i}_n(v)$,
  $\mathrm{t}_n(u)=\mathrm{t}_n(v)$, and\/ \pv V~satisfies the
  pseudoidentity $\Phi_n(u)=\Phi_n(v)$.
\end{Thm}

Note that, if \pv V contains the pseudovariety $\pv{RB}$, then the
assumption that \pv V~satisfies $\Phi_n(u)=\Phi_n(v)$ implies that, either
$u,v\in A_{\le n}$, or $\Phi_n(u)$ and $\Phi_n(v)$ start and end with
the same letters, which automatically guarantees the other two
conditions in the theorem, namely $\mathrm{i}_n(u)=\mathrm{i}_n(v)$ and
$\mathrm{t}_n(u)=\mathrm{t}_n(v)$. By
Theorem~\ref{t:word-problem-for-V*Dn}, the pseudovariety \pv{RB} is
contained in~$\pv{Sl}*\pv D_1$, and the latter is contained in many of
the pseudovarieties in which we are interested in this paper which,
moreover, satisfy no nontrivial identities. For this reason, we will
usually omit reference to the conditions
$\mathrm{i}_n(u)=\mathrm{i}_n(v)$ and
$\mathrm{t}_n(u)=\mathrm{t}_n(v)$ when applying
Theorem~\ref{t:word-problem-for-V*Dn}.
The assumption $\pv{Sl}*\pv D_1\subseteq V$ also
gives the inclusion 
$\pv{Sl}\subseteq \pv V$ which implies that \pv V contains 
a nontrivial monoid.

Another observation regarding Theorem~\ref{t:word-problem-for-V*Dn},
which is formulated below as Lemma~\ref{l:relative-conditions-for-V*Dn}, is
that, if $\pv V*\pv D_n=\pv V$ and \pv V contains~\pv{Sl}, then the
mapping $\Phi_n$~induces a function $\Phi_n^\pv V:\Om
AV\to(\Om{A_{n+1}}V)^1$ that also satisfies properties
\eqref{item:Phi-a}--\eqref{item:Phi-c}. 
We clarify some technicalities before we state the lemma formally.
First,  the equality
$\pv V*\pv D_n=\pv V$ implies  $\pv D_n\subseteq \pv V$
and we may assume that $A_{\le n}\subseteq\Om AV$.
Further, for $w\in A_{n+1}\subseteq  \Om AS$ and 
$u\in \Om AS$ such that $w=u$ holds in $\pv V*\pv D_n=\pv V$,
the pseudoidentity $\Phi_n(u)=\Phi_n(w)$ also holds in \pv V 
by Theorem~ \ref{t:word-problem-for-V*Dn}.
Since $\Phi_n(w)=w\in A_{n+1}$ and $\pv D_2\subseteq \pv V$
we get  $\Phi_n(u)=w$ in $\Om {A_{n+1}}S$, which entails the equality $u=w$.
Altogether, we may assume that $A_{\le n+1}$ is embedded in $\Om AV$
and $\pi^{-1}(w)=\{w\}$ for $w\in A_{\le n+1}$ 
and the natural projection $\pi: \Om AS \rightarrow \Om AV$.

\begin{Lemma}\label{l:relative-conditions-for-V*Dn}
 Let $n>0$ and consider a pseudovariety \pv V that contains %
  $\pv V*\pv D_n$ and~\pv{Sl}. Then there exists a continuous function
$\Phi_n^{\pv V}$ such that the following diagram commutes, where the
vertical arrows are the natural projections:
\begin{equation}
  \label{eq:Phi-nV}
  \vcenter{
    \vbox{
      \xymatrix{
        \Om AS %
        \ar[r]^(.4){\Phi_n} %
        \ar[d]_\pi %
        & %
        (\Om{A_{n+1}}S)^1 %
        \ar[d]_{\sigma_n} %
        \\ %
        \Om AV %
        \ar[r]^(.4){\Phi_n^{\pv V}} %
        & %
        (\Om{A_{n+1}}V)^1. %
      }}}
\end{equation}
Moreover,  the following properties hold:
\begin{enumerate}[(a)]
\item\label{item:PhiV-a} $\Phi_n^{\pv V}(w)=1$ for every $w\in A_{\le n}$;
\item\label{item:PhiV-b} $\Phi_n^{\pv V}(w)=w$ for $w\in A_{n+1}$;
\item\label{item:PhiV-c} $\Phi_n^{\pv V}(uv) %
  =\Phi_n^{\pv V}(u)\Phi_n^{\pv V}(\mathrm{t}_n(u)\,v) %
  =\Phi_n^{\pv V}(u\,\mathrm{i}_n(v))\Phi_n^{\pv V}(v)$ %
  for all $u,v\in\Om AV$.
\end{enumerate}
\end{Lemma}

\begin{proof}
  If $u,v\in\Om AS$ are such that $\pi(u)=\pi(v)$, then the
  pseudoidentity $u=v$ holds in~$\pv V=\pv V*\pv D_n$. By
  Theorem~\ref{t:word-problem-for-V*Dn}, it follows that so does the
  pseudoidentity $\Phi_n(u)=\Phi_n(v)$, whence the equality
  $\sigma_n(\Phi_n(u))=\sigma_n(\Phi_n(v))$ holds. Thus, there is a
  function $\Phi_n^\pv V$ such that the diagram commutes. It is
  continuous because so are $\sigma_n$, $\Phi_n$, and $\pi$, and \Om
  AS is compact. The verification of properties~\eqref{item:Phi-a}
  and~\eqref{item:Phi-b} for~$\Phi_n^\pv V$ is immediate, while
  property~\eqref{item:Phi-c} follows from the commutativity of the
  diagram~\eqref{eq:Phi-nV} and the fact that $\pi$~is surjective.
\end{proof}

Although the function $\Phi_n^\pv V$ is not a homomorphism, we may
prove the following consequence of
Theorem~\ref{t:word-problem-for-V*Dn}, which states that $\Phi_n^\pv
V$ provides a rather convenient means of encoding \Om AV
in~\Om{A_{n+1}}V, which we call the \emph{de Bruijn encoding}.

\begin{Thm}
  \label{t:Phi-vs-Green}
  Let $n>0$ and consider a pseudovariety \pv V that contains %
  $\pv V*\pv D_n$ and~\pv{Sl}. Then the mapping $\Phi_n^\pv V$ is
  injective on $\Om AV\setminus A_{\le n}$. Moreover, for $u,v\in\Om
  AV\setminus A_{\le n}$ and any of Green's equivalence relations \Cl
  K, $u$ and $v$ are \Cl K-related in~\Om AV if and only if so are
  $\Phi_n^\pv V(u)$ and $\Phi_n^\pv V(v)$ in~$\Om{A_{n+1}}V$.
\end{Thm}

\begin{proof}
  Given $w,z\in\Om AS\setminus A_{\le n}$, 
 since the diagram~\eqref{eq:Phi-nV} commutes,
  the equality $\Phi_n^\pv V(\pi(w))=\Phi_n^\pv V(\pi(z))$ is
  equivalent to the pseudoidentity $\Phi_n(w)=\Phi_n(z)$ being valid
  in~\pv V. By Theorem~\ref{t:word-problem-for-V*Dn}, this in turn is
  equivalent to the pseudoidentity $w=z$ being valid in $\pv V*\pv
  D_n=\pv V$, that is $\pi(w)=\pi(z)$. Hence, the restriction of
  $\Phi_n^\pv V$ to $\Om AS\setminus A_{\le n}$ is
  injective.

  The statement about Green's equivalence relations is handled
  similarly for all of them. Consider, for instance the \Cl
  R-ordering.
  
  Suppose that $u\ge_\Cl Rv$ in $\Om AV$, 
  which means that there is some
  $w\in(\Om AV)^1$ such that $uw=v$. Applying $\Phi_n^\pv V$ and
  taking into account property~\eqref{item:Phi-c}, we
  obtain $\Phi_n^\pv V(v)=\Phi_n^\pv V(u)\Phi_n^\pv
  V(\mathrm{t}_n(u)w)$, which shows that $\Phi_n^\pv V(u)\ge_\Cl
  R\Phi_n^\pv V(v)$ in $(\Om{A_{n+1}}V)^1$.

  Conversely, suppose that $\Phi_n^\pv V(u)\ge_\Cl R\Phi_n^\pv V(v)$,
  that is $\Phi_n^\pv V(u)\,t=\Phi_n^\pv V(v)$ for some
  $t\in(\Om{A_{n+1}}V)^1$. Recall that $u,v\in\Om AV\setminus A_{\le
    n}$. Since \pv V contains $\pv{Sl}*\pv D_1$, the pseudowords
  $\Phi_n^\pv V(u)\,t$ and $\Phi_n^\pv V(v)$ must have exactly the
  same factors of length~$2$. From the definition of~$\Phi_n^\pv V$,
  it follows that all factors of length~$2$ of~$\Phi_n^\pv V(u)\,t$
  must be of the form $(ax)(xb)$, where $x\in A_n$ and $a,b\in A$.
  Since this is precisely the condition that characterizes membership
  in the image of the function~$\Phi_n^\pv V$, it follows that
  $t=\Phi_n(\mathrm{t}_n(u)w)$ for some $w\in\Om AV$. In view of
  property~\eqref{item:Phi-c} of the function~$\Phi_n^\pv V$, it
  follows that $\Phi_n^\pv V(v)=\Phi_n^\pv V(u)\,t=\Phi_n^\pv V(uw)$.
  Since $\Phi_n^\pv V$ is injective by the first part of the proof, we
  deduce that $uw=v$, which shows that $u\ge_\Cl Rv$.
\end{proof}

\subsection{Content and related functions}
\label{sec:content-0-1}

Let $S$ be a topological semigroup. For a subset $X$ of~$S$, denote by
$\overline{\langle X\rangle}$ the closed subsemigroup generated
by~$X$. For a finite set $A$, we say that $S$~is \emph{$A$-generated}
if there is a mapping $\varphi:A\to S$ such that
$\overline{\langle\varphi(A)\rangle}=S$. Usually, the generating
function $\varphi$ will be understood from the context and not
mentioned explicitly. Moreover, whenever we use a letter $a\in A$ to
represent an element of~$S$, we really mean the element~$\varphi(a)$.

We say that the $A$-generated profinite semigroup $S$~has a
\emph{content function} if the natural projection $\Om AS\to\Om A{Sl}$
factorizes through the unique extension of~$\varphi$ to a continuous
homomorphism $\hat\varphi:\Om AS\to S$. Equivalently, for subsets $B$
and $C$ of~$A$, if $s\in S$ belongs to both $\overline{\langle
  B\rangle}$ and $\overline{\langle C\rangle}$, then $B=C$. Then, for
each $s\in S$, the unique subset $B$ of~$A$ such that
$s\in\overline{\langle B\rangle}$ is denoted $c(s)$ and is called the
\emph{content} of~$s$.

Suppose that $S$ has a content function. For $s\in S$, we denote by
$0(s)$ the set of all $t\in S^1$ such that there is a factorization
$s=tas'$ with $c(s)=c(t)\uplus\{a\}$. The set of all such $a\in A$ is
also denoted $\bar0(s)$. Dually, the set of all $t\in S^1$ such that
there is a factorization $s=s'at$ with $c(s)=c(t)\uplus\{a\}$ is
denoted $1(s)$, and $\bar1(s)$~is defined similarly.
Following~\cite[Section~3]{Almeida&Trotter:1999a}, we say that
$S$\emph{~has $0$, $\bar 0$, $1$, and $\bar1$ functions} if,
respectively, each of the sets $0(s)$, $\bar0(s)$, $1(s)$, and
$\bar1(s)$ is a singleton for every $s\in S$. Such singleton sets will
be identified with their unique elements. Note that if $S$~has $0$ and
$\bar0$ functions then, by iterating these functions, we conclude
that, for $s\in S$, the order in which generators occur in~$s$ for the
first time from left to right is well determined, and so are the
prefixes determined by those first occurrences.

Let $f$ be one of the functions content, $0$, $\bar0$, $1$, or
$\bar1$. We say that a pseudovariety \pv V \emph{has the function $f$}
if so does the semigroup \Om AV for every finite alphabet~$A$.

There are many pseudovarieties \pv V which have content,
$0$, and $\bar0$ functions. A sufficient condition is given
in~\cite[Proposition~3.5]{Almeida&Trotter:1999a}: it suffices that the
pseudovariety \pv V contain~\pv{Sl} and be closed under taking right
Rhodes expansions (cut down to generators). In turn, a simple
sufficient condition for a pseudovariety \pv V to be closed under
right Rhodes expansions is that $\pv{LZ}\malcev\pv V=\pv V$
\cite{Reilly:1990}. Dually, \pv V is closed under left Rhodes
expansions if $\pv{RZ}\malcev\pv V=\pv V$. Obvious sufficient
conditions for the conjunction of the conditions $\pv{LZ}\malcev\pv
V=\pv V$ and $\pv{RZ}\malcev\pv V=\pv V$ are that $\pv{RB}\malcev\pv
V=\pv V$ or $\pv B\malcev\pv V=\pv V$. Iterating alternately right and left 
Rhodes expansions on a finite semigroup $S$, the process stops (up to
isomorphism) in a finite number of steps. The resulting semigroup is
known as the \emph{Birget expansion} of~$S$
\cite{Birget:1984,Birget&Rhodes:1984}.

\subsection{Equidivisibility and complexity}
\label{sec:equidivisibility}

A semigroup $S$~is said to be \emph{equidivisible} if, whenever
$s,t,u,v$ are elements of~$S$ such that $st=uv$, there exists some
$w\in S^1$ such that either $u=sw$ and $wv=t$, or $uw=s$ and $v=wt$.
This notion was introduced in~\cite{McKnight&Storey:1969}, as a
generalization of free semigroup. A pseudovariety of semigroups \pv V
is also said to be \emph{equidivisible} if \Om AV is equidivisible for
every finite set~$A$.

It is easy to show that every equidivisible pseudovariety containing
\pv{Sl} has $0$, $\bar0$, $1$, and $\bar1$ functions.

A sufficient condition for equidivisibility has been explicitly given
in~\cite{Almeida&ACosta:2007a}. We say that a pseudovariety \pv V is
\emph{closed under concatenation} if the variety of regular languages
corresponding to it according to Eilenberg's 
correspondence \cite{Pin:1986;bk} enjoys
that property, that is, if $K$ and $L$ are languages over the same
finite alphabet whose syntactic semigroups belong to~\pv V, then so
does the syntactic semigroup of the language $KL$. It is proved
in~\cite[Lemma~4.8]{Almeida&ACosta:2007a} that every pseudovariety
closed under concatenation is equidivisible.
In particular, $\pv S$ is  equidivisible.

The \emph{closure under concatenation} of a pseudovariety \pv V is the
smallest pseudovariety closed under concatenation that contains~\pv V.
It may be described as $\pv A\malcev\pv V$
\cite{Straubing:1979a,Chaubard&Pin&Straubing:2006}. Hence, a
pseudovariety \pv V is closed under concatenation if and only if it
satisfies the equation %
$\pv A\malcev\pv V=\pv V$.

For a pseudovariety \pv V containing \pv N, the property of being
closed under concatenation also has a very simple and useful
topological formulation. Namely, it is equivalent to the
multiplication of~\Om AV being an open mapping for every finite
set~$A$ \cite[Lemma~2.3]{Almeida&ACosta:2007a}.

A familiar class of examples of pseudovarieties closed under
concatenation is given by the pseudovarieties of the form $\bar{\pv
  H}$, where \pv H~is an arbitrary pseudovariety of groups. It is
indeed a simple exercise to check that %
$\pv A\malcev\bar{\pv H}=\bar{\pv H}$. Note also that $\bar{\pv H}*\pv
A=\bar{\pv H}$.

Another example is given by the Krohn-Rhodes complexity
pseudovarieties $\pv C_n$, which are extensively studied
in~\cite[Chapter~4]{Rhodes&Steinberg:2009qt}. They are defined
recursively by $\pv C_0=\pv A$ and $\pv C_{n+1}=\pv C_n*\pv G*\pv A$.
By~\cite[Corollary~4.9.4]{Rhodes&Steinberg:2009qt}, the equality %
$\pv A\malcev\pv C_n=\pv C_n$ holds for every $n\ge0$. Another
property of interest for the purposes of this paper is that the
complexity pseudovarieties $\pv C_n$ are monoidal
by~\cite[Proposition~4.3.14]{Rhodes&Steinberg:2009qt}. Thus, we have
the following result, which we state here for later reference.

\begin{Prop}
  \label{p:Cn-properties}
  Let \pv H be a pseudovariety of groups, $n\ge0$, and \pv V be one of
  the pseudovarieties $\bar{\pv H}$ and $\pv C_n$. Then \pv V is
  monoidal and closed under concatenation, it has content, $0$,
  $\bar0$, $1$, and $\bar1$ functions, and the equalities %
  $\pv V*\pv D=\pv D\malcev\pv V=\pv K\malcev\pv V=\pv V$ %
  hold.\qed
\end{Prop}

\subsection{Letter cancelation}
\label{sec:letter-cancelation}

We say that an $A$-generated topological semigroup $S$ is \emph{right
  letter cancelative} if, for every generator $a\in A$ and all $s,t\in
S$, if $sa=ta$ then $s=t$. The pseudovariety \pv V is said to be
\emph{right letter cancelative} if so is each semigroup \Om AV for
every finite alphabet~$A$. Equivalently, if \pv V satisfies the
pseudoidentity $ua=va$ over a finite alphabet $A$, where $a\in A$,
then it also satisfies the pseudoidentity $u=v$. The dual notion of
right letter cancelative is \emph{left letter cancelative}, whose
precise definition for a topological semigroup and for a pseudovariety
is left to the reader.

The following result assumes familiarity with Eilenberg's
correspondence between pseudovarieties of semigroups and varieties of
languages. The proof can be considered an exercise in the theory of
profinite semigroups, but is included for the sake of completeness.
The reader may wish to recall that the variety of languages \Cl V
corresponding to a pseudovariety of semigroups \pv V associates with a
finite alphabet $A$ the set $\Cl V(A)$ of all \pv V-recognizable
subsets of~$A^+$, that is subsets that can be recognized by
homomorphisms from $A^+$ into semigroups from~\pv V. The proof below
uses mainly the fact that the topological space \Om AV is the Stone
dual of the Boolean algebra $\Cl V(A)$
\cite[Theorem~3.6.1]{Almeida:1994a}, a fact that is referred in the
proof simply as ``Stone duality''. This duality may be expressed as
follows, where $\iota:A^+\to\Om AV$ is the natural homomorphism: a
language $L\subseteq A^+$ belongs to $\Cl V(A)$ if and only if
$\overline{\iota(L)}$ is open in~\Om AV and
$\iota^{-1}\bigl(\overline{\iota(L)}\bigr)=L$; furthermore, the sets
$\overline{\iota(L)}$ suffice to separate points of~\Om AV. Moreover,
in case \pv V contains \pv N, the mapping $\iota$ is injective, the
induced topology on $A^+$ is discrete, and the condition
$\iota^{-1}\bigl(\overline{\iota(L)}\bigr)=L$ is superfluous
\cite[Theorem~2.12]{Almeida&Costa:2015hb}. Furthermore, the clopen
sets of the form $\overline{\iota(L)}$ are sufficient to separate
points of~\Om AV, and so they generate the topology of~\Om AV.

\begin{Prop}\label{p:cancelation-in-terms-of-languages}
  Let \pv V be a pseudovariety of semigroups, \Cl V be the
  corresponding variety of languages. Then, the following conditions
  are equivalent:
  \begin{enumerate}
  \item\label{item:p:cancelation-in-terms-of-languages-1} for every
    finite alphabet $A$ and every letter $a\in A$, $L\in\Cl V(A)$
    implies $La\in\Cl V(A)$;
  \item\label{item:p:cancelation-in-terms-of-languages-2} the
    pseudovariety \pv V contains \pv D and it is right letter
    cancelative;
  \item\label{item:p:cancelation-in-terms-of-languages-3} the
    pseudovariety \pv V contains \pv{RZ} and it is right letter
    cancelative.
  \end{enumerate}
\end{Prop}

\begin{proof}
  $\eqref{item:p:cancelation-in-terms-of-languages-1}\Rightarrow
  \eqref{item:p:cancelation-in-terms-of-languages-2}$ Let $A$~be a
  finite alphabet and let $u,v\in\Om AV$ and $a\in A$ be such that
  $ua=va$. Assuming
  \eqref{item:p:cancelation-in-terms-of-languages-1}, we show that the
  inequality $u\ne v$ leads to a contradiction.

  Let $\iota:A^+\to\Om AV$ be the natural homomorphism. Assuming that
  $u\ne v$, by Stone duality there is a \pv V-recognizable language
  $L\subseteq A^+$ such that $\overline{\iota(L)}$ contains $u$ but
  not $v$. Consider a sequence of words $(v_n)_n$ from~$A^+$ such that
  $\lim\iota(v_n)=v$. It follows that $\lim\iota(v_na)=va=ua$. By
  hypothesis, the language $La$ belongs to~$\Cl V(A)$, which, by Stone
  duality, entails that the set $\overline{\iota(La)}$~is open and
  $\iota^{-1}(\overline{\iota(La)})=La$. Since
  $ua\in\overline{\iota(La)}$, the words $v_na$~must belong to~$La$
  for all sufficiently large~$n$. Hence, $v_n$~lies in $L$ for all
  sufficiently large~$n$, so that $v\in\overline{\iota(L)}$, which
  contradicts the assumption that~$v\notin L$.

  Hence, \pv V is right letter cancelative. That \pv V contains \pv D
  also follows from~\eqref{item:p:cancelation-in-terms-of-languages-1}
  can be seen by iterating the operations $L\mapsto La$ on $A^+$,
  since the variety of languages corresponding to~\pv D consists of
  all languages that are finite Boolean combinations of languages of
  the form %
  $A^*u=A^+u\cup a^{-1}A^+u$ ($a\in A$), with $u$ a finite word.

  $\eqref{item:p:cancelation-in-terms-of-languages-3}\Rightarrow
  \eqref{item:p:cancelation-in-terms-of-languages-1}$ Suppose that
  $\pv V\supseteq\pv{RZ}$ and that \pv V is right letter cancelative.
  Let $L$ be a language from~$\Cl V(A)$. We show that the condition
  $La\notin\Cl V(A)$ leads to a contradiction. By Stone duality,
  $\overline{\iota(L)}$~is an open subset of~\Om AV but
  $\overline{\iota(La)}=\overline{\iota(L)}a$~is not.
  The latter condition implies that there is $u\in\overline{\iota(L)}$
  such that $ua=\lim w_n$ for a sequence $(w_n)_n$ of words in
  $A^+\setminus La$. Since \pv V contains~\pv{RZ}, the function
  $\mathrm{t}_1$ is well defined on~\Om AV and it is continuous.
  Hence, we may assume that, for all~$n$, there is a factorization
  $w_n=v_na$. By compactness of~\Om AV, we may further assume that the
  sequence $(v_n)_n$ converges to some $v\in\Om AV$. It follows that
  $ua=\lim w_n=\lim v_na=va$ which, since \Om AV is assumed to be
  right letter cancelative, yields the equality $u=v$. As $u$ was
  chosen as an element of the open set $\overline{\iota(L)}$ and $\lim
  v_n=v=u$, we deduce that $v_n\in L$, whence also $w_n\in La$ for all
  sufficiently large $n$, which contradicts the choice of the
  sequence~$(w_n)_n$.
\end{proof}

The language closure property
\eqref{item:p:cancelation-in-terms-of-languages-1} of
Proposition~\ref{p:cancelation-in-terms-of-languages} is not necessary
for right letter cancelativity. For example, using the structure
theorem for \Om AJ \cite[Theorem~8.2.8]{Almeida:1994a}, one may show
that \pv J is right letter cancelative.

In view of Proposition~\ref{p:cancelation-in-terms-of-languages}, an
obvious sufficient condition for a pseudovariety to be both left and
right cancelative is that the corresponding variety of languages be
closed under concatenation. For our purposes, we need an alternative
sufficient condition, which can be obtained by taking into account some
results from~\cite{Straubing:1985}.

\begin{Prop}
  \label{p:cancellation}
  Let \pv V be a nontrivial monoidal pseudovariety of semigroups such
  that $\pv V*\pv D=\pv V$. Then \pv V~is both left and right letter
  cancelative.
\end{Prop}

\begin{proof}  
  By \cite[Corollary~3.3]{Straubing:1985}, the equality %
  $\pv V*\pv{LI}=\pv V$ holds. Suppose that the syntactic semigroup
  $S(L)$ of the language $L\subseteq A^+$ belongs to~\pv V and let
  $a\in A$ be a letter. By \cite[Lemma~9.8]{Straubing:1985},
  $S(La)$~also belongs to~\pv V. We may therefore apply
  Proposition~\ref{p:cancelation-in-terms-of-languages}
  to deduce that \pv V~is right letter cancelative.

  To complete the proof, we show that the dual $\pv V^\rho$ of~\pv V,
  which is clearly also nontrivial and monoidal, is again such that
  $\pv V^\rho*\pv D=\pv V^\rho$. Indeed, from the equality $\pv
  V*\pv{LI}=\pv V$ we obtain $\pv V^\rho=(\pv V*\pv{LI})^\rho=\pv
  V^\rho*\pv{LI}$, where the second equality is given
  by~\cite[Proposition~4.4]{Straubing:1985}.
\end{proof}

\subsection{Basic factorizations}
\label{sec:basic-factorizations}

In this section, we consider a strengthening of the property of a
pseudovariety to have content and $0$ and $\bar0$ (or $1$ and $\bar1$)
functions.

By a \emph{left basic factorization} of an element
$s$ of a semigroup $S$ with a content function~$c$, we mean a
factorization of the form $s=s_0as_1$ with $s_0,s_1\in S^1$ such that
$c(s)=c(s_0)\uplus\{a\}$. In such a factorization, the generator
$a$~is said to be the \emph{marker} and $s_1$ the \emph{remainder}. We
say $S$ has \emph{unique left basic factorizations} if, for any two
left basic factorizations $s=s_0as_1$ and $s=t_0bt_1$ of the same
element, we have $s_0=t_0$, $a=b$, and $s_1=t_1$. Given a generator
$a$ and an element $s$ of a semigroup $S$ with unique left basic
factorizations, the first occurrence of $a$, from left to right, as a
factor of~$s$ can be located by iterated left basic factorization on
the left factor until it is found as a marker. The factor that follows
it is called the \emph{absolute remainder} of $a$ in~$s$. Iterating
this procedure on the absolute remainders, one may successively locate
first occurrences of the letters of any word $u=a_1\cdots a_r$ on the
generators for which there is a factorization $s=s_0a_1s_1\cdots
a_rs_r$ of a given element of~$S$. This is called the
\emph{left-greedy occurrence} of $u$ as a subword in~$s$, and $s_r$~is
called its \emph{remainder}.

We say that a pseudovariety \pv V has \emph{unique left basic
  factorizations} if, for every finite alphabet $A$, \Om AV has unique
left basic factorizations.

The definition of \emph{right basic factorizations} and the property
of having \emph{unique right basic factorizations} are the left-right
duals of the above notions.

Combining the unilateral version
of~\cite[Proposition~3.4]{Almeida:1996c} with the characterization of
pseudovarieties whose corresponding varieties of languages are closed
under deterministic product \cite{Pin:1980b}, we obtain the following
sufficient condition for uniqueness of left basic factorizations at
the pseudovariety level and its dual.

\begin{Prop}
  \label{p:unique-lbf}
  Let\/ \pv V be a monoidal pseudovariety containing \pv{Sl}. If\/ \pv
  V satisfies the equation $\pv D\malcev\pv V=\pv V$, then \pv V has
  unique left basic factorizations. Dually, if\/ \pv V satisfies the
  equation $\pv K\malcev\pv V=\pv V$, then \pv V has unique right
  basic factorizations.
\end{Prop}

It can be easily checked that many familiar examples of
pseudovarieties \pv V satisfy the equation %
$\pv D\malcev\pv V=\pv V$. Two families of such examples are
registered in the following result.

\begin{Cor}
  \label{c:unique-lbf-egs}
  Let \pv H be an arbitrary pseudovariety of groups. Then the
  pseudovarieties $\pv{DO}\cap\bar{\pv H}$ and $\pv{DS}\cap\bar{\pv
    H}$ have unique left and right basic factorizations.\qed
\end{Cor}

Another application of Proposition~\ref{p:unique-lbf} is obtained by
invoking Proposition~\ref{p:Cn-properties}.

\begin{Cor}
  \label{c:unique-bf-Cn}
  The pseudovarieties $\pv C_n$ have unique left and right basic
  factorizations and so do the pseudovarieties $\pv{DS}\cap\pv
  C_n$.\qed
\end{Cor}

Since the product of a letter by a language is always deterministic,
we also have the following immediate consequence of the results from
\cite{Pin:1980b}. Alternatively, one may easily show, by iterating on
the left factors left basic factorizations, that a pseudovariety with
unique left basic factorizations is left letter cancelative.

\subsection{Some special examples}
\label{sec:some-special-examples}

The aim of this subsection is to prove some auxiliary results which
provide examples of pseudovarieties for which a result in
Section~\ref{sec:wggm-sufficient-conditions} applies.

\begin{Prop}
  \label{p:DAmV-vs-starD}
  Let \pv V~be a monoidal pseudovariety containing \pv{Sl} such that
  $\pv V*\pv D=\pv V$. Then the pseudovariety $\pv W=\pv{DA}\malcev\pv
  V$~is such that $\pv W*\pv D=\pv W$.
\end{Prop}

\begin{proof}
  By Theorem~\ref{t:word-problem-for-V*Dn}, a pseudoidentity $u=v$
  holds in~$\pv W*\pv D_n$ if and only if \pv W satisfies the
  pseudoidentity $\Phi_n(u)=\Phi_n(v)$. On the other hand, by the
  Basis Theorem for Mal'cev products
  \cite[Theorem~4.1]{Pin&Weil:1996b}, \pv W~is defined by the
  pseudoidentities of the form $((uv)^\omega u)^2=(uv)^\omega u$,
  where $u$ and $v$ are pseudowords such that the pseudoidentities
  $u=v=v^2$ hold in~\pv V. Thus, to show that $\pv W*\pv D_n$ is
  contained in~\pv W, we assume that the pseudoidentities $u=v=v^2$
  hold in~\pv V and we need to prove that the pseudoidentity
  \begin{equation}
    \label{eq:DAmV-0}
    \Phi_n\bigl(((uv)^\omega u)^2\bigr) %
    =\Phi_n\bigl((uv)^\omega u\bigr)
  \end{equation}
  holds in~\pv W.
  As $\pv D\subseteq \pv V*\pv D=\pv V$, we must have 
  $u,v\in \Om AS\setminus A^+$.

  Since $\pv V*\pv D_n=\pv V$, from
  Theorem~\ref{t:word-problem-for-V*Dn} we deduce that
  $\mathrm{i}_n(u)=\mathrm{i}_n(v)$,
  $\mathrm{t}_n(u)=\mathrm{t}_n(v)$, and the pseudoidentities
  \begin{equation}
    \label{eq:DAmV-1}
    \Phi_n(u)=\Phi_n(v)=\Phi_n(v^2)
  \end{equation}
  hold in~\pv V. Consider the word
  $s=\Phi_n(\mathrm{t}_n(u)\mathrm{i}_n(u))$. Taking into account
  property~\eqref{item:Phi-c}
  of the function~$\Phi_n$, we may express
  $\Phi_n(v^2)$ as the product $\Phi_n(v)s\Phi_n(v)$. Multiplying on
  the right all sides of the pseudoidentities
  $\Phi_n(u)=\Phi_n(v)=\Phi_n(v^2)$ by $s$, we deduce that the
  pseudowords $u'=\Phi_n(u)s$ and $v'=\Phi_n(v)s$ are such that the
  pseudoidentities $u'=v'=(v')^2$ hold in~\pv V. Hence, the
  pseudoidentity %
  $((u'v')^\omega u')^2=(u'v')^\omega u'$ holds in~\pv W. Using again
  property~\eqref{item:Phi-c} of the function $\Phi_n$, we deduce
  that, for the pseudoidentities
  \begin{equation}
    \label{eq:DAmV-2}
    \Phi_n\bigl(((uv)^\omega u)^2\bigr)s %
    =\bigl(\Phi_n\bigl((uv)^\omega u\bigr)s\bigr)^2 %
    =\Phi_n\bigl((uv)^\omega u\bigr)s,
  \end{equation}
  the first is valid in every finite semigroup, while the second holds
  in~\pv W. Since %
  $\pv{LI}\malcev\pv W\subseteq\pv{DA}\malcev\pv W=\pv W$, we know
  from Proposition~\ref{p:unique-lbf} that \pv W is right
  letter cancelative. Hence, from the fact \pv W~satisfies the
  pseudoidentities \eqref{eq:DAmV-2}, it follows that \pv W~also
  satisfies the pseudoidentity \eqref{eq:DAmV-0}.
\end{proof}

\begin{Cor}
  \label{c:DAm.VstarA-properties}
  Let \pv V be a monoidal pseudovariety of semigroups. Then the
  pseudovariety $\pv W=\pv{DA}\malcev(\pv V*\pv A)$ has the following
  properties:
  \begin{enumerate}[(i)]
  \item\label{item:DAm.VstarA-1} \pv W is monoidal;
  \item\label{item:DAm.VstarA-2} \pv W is both left and right letter
    cancelative;
  \item\label{item:DAm.VstarA-3} $\pv W*\pv D=\pv W$;
  \item\label{item:DAm.VstarA-4} $\pv B\malcev\pv W=\pv W$.
  \end{enumerate}
\end{Cor}

\begin{proof}
  By the remarks at the end of
  Subsection~\ref{sec:pvs-and-operations}, we obtain that both $\pv
  V*\pv A$ and \pv W are monoidal, the latter being
  property~\eqref{item:DAm.VstarA-1}. The pseudovariety $\pv V*\pv A$
  certainly contains \pv{Sl}, as so does \pv A. Moreover, as %
  $\pv{LI}\subseteq\pv{DA}$,
  property~\eqref{item:DAm.VstarA-2}~follows from
  Proposition~\ref{p:unique-lbf}. Since the semidirect product
  is associative and $\pv A*\pv D=\pv A$, we have %
  $(\pv V*\pv A)*\pv D=\pv V*\pv A$. Invoking
  Proposition~\ref{p:DAmV-vs-starD}, we obtain
  property~\eqref{item:DAm.VstarA-3}. Finally,
  property~\eqref{item:DAm.VstarA-4} follows from the inclusion %
  $\pv B\subseteq\pv{DA}$.
\end{proof}

\section{Translational representations}
\label{sec:translation-representations}

We introduce in this section certain representations of profinite
semigroups given by translational action on the minimum ideal. They
are explored in this paper to derive the applications in
Sections~\ref{sec:orderability} and~\ref{sec:irreducibility}. It is
hoped however that they will eventually also shed light on the
structure of relatively free profinite semigroups.

\subsection{The translational hull}
\label{sec:translational-hull}

This subsection is partly based
on~\cite[Chapter~4]{Carruth&Hildebrant&Koch:1986}.

For topological spaces $X$ and $Y$, denote by $C(X,Y)$ the space of
all continuous functions $X\to Y$. A net $(f_i)_i$ in $C(X,Y)$ is said
to \emph{converge continuously} to~$f\in C(X,Y)$ if, for every net
$(x_j)_j$ in~$X$ with limit $x$, the net
$\bigl(f_i(x_j)\bigr)_{(i,j)}$ converges to~$f(x)$. The following
lemma relates continuous convergence with convergence in the
compact-open topology. It is essentially the same
as~\cite[Lemma~4.1]{Carruth&Hildebrant&Koch:1986}.

\begin{Lemma}
  \label{l:continuous-convergence-in-locally-compact}
  Let $X$~be a locally compact space. Then a net $(f_i)_i$ converges
  continuously to $f$ in $C(X,Y)$ if and only if it converges to~$f$
  in the compact-open topology of~$C(X,Y)$.
\end{Lemma}

For a topological space $X$, the set $C(X,X)$ is a monoid under
composition, which is denoted $\Cl{T}_X^\ell$ or $\Cl{T}_X^r$
according to whether functions are taken to act and are composed on
the left or on the right, respectively. These function spaces are
endowed with the compact-open topology. In case $X$ is locally
compact, it follows from
Lemma~\ref{l:continuous-convergence-in-locally-compact} that
$\Cl{T}_X^\ell$ and $\Cl{T}_X^r$ are topological semigroups, in the
sense that multiplication is continuous. More generally, the
continuity of composition follows
from~\cite[Theorem~3.4.2]{Engelking:1989} which states that, for
topological spaces $X,Y,Z$, the mapping from $C(Y,Z)\times C(X,Y)$ to
$C(X,Z)$ given by the formula $(f,g)\mapsto f\circ g$ is continuous
under the assumption that $Y$ is locally compact.

Let $S$~be a topological semigroup. A \emph{left translation} of~$S$
is a mapping $\lambda\in \Cl{T}_S^\ell$ such that
$\lambda(st)=\lambda(s)t$ for all $s$ and $t$ in~$S$.
Dually, a
\emph{right translation} is a mapping $\rho\in \Cl{T}_S^r$ such that
$(st)\rho=s(t)\rho$ whenever $s,t\in S$. The \emph{inner left
  translation} of~$S$ determined by an element $s\in S$ is the mapping
$\lambda_s\in \Cl{T}_S^\ell$ defined by $\lambda_s(u)=su$. The
\emph{inner right translation} $\rho_s$ determined by $s$ is defined
dually. 

The mappings $\lambda\in \Cl{T}_S^\ell$ and $\rho\in \Cl{T}_S^r$ are
\emph{linked} if $s\,\lambda(t)=(s)\rho\,t$ for all $s,t\in S$. A
\emph{bitranslation} of~$S$ is a linked pair $(\lambda,\rho)$ in which
$\lambda$ is a left translation and $\rho$~is a right translation.
Note that, for $s\in S$, the pair $\omega_s=(\lambda_s,\rho_s)$ is a
bitranslation, which is called the \emph{inner bitranslation}
determined by~$s$. Note also that the pair in which both components
are the identity function on $S$ is a bitranslation. The
\emph{translational hull} $\Omega(S)$ of~$S$ consists of all
bitranslations of~$S$. Note that $\Omega(S)$~is a closed submonoid of
the product $\Cl{T}_S^\ell\times\Cl{T}_S^r$. In particular, if $S$ is
a locally compact semigroup, then $\Omega(S)$~is a topological monoid.
Its topology as a subspace of~$\Cl{T}_S^\ell\times\Cl{T}_S^r$ is the
compact-open topology. The space $\Omega(S)$ may also be viewed as a
space of continuous functions, namely as a subset of~$C(S,S\times S)$.

A semigroup $S$ is \emph{right reductive} if the \emph{canonical
  mapping} $S\to \Cl{T}_S^\ell$ sending each $s\in S$ to the inner
left translation $\lambda_s$ is injective. A \emph{left reductive}
semigroup is defined dually. The semigroup $S$ is \emph{reductive} if
it is both left and right reductive. We also say that $S$ is
\emph{weakly reductive} if the \emph{canonical mapping}
$S\to\Omega(S)$ sending each $s\in S$ to~$\omega_s$ is injective; its
image is an ideal of $\Omega(S)$
\cite[Corollary~1.11]{Hildebrant&Lawson&Yeager:1976}. Note that every
monoid is reductive.

\begin{Thm}[{\cite[Corollary~4.7 and
    Theorem~4.9]{Carruth&Hildebrant&Koch:1986}}]
  \label{t:TH-compact-reductive}
  Let $S$ be a compact reductive semigroup $S$. Then the compact-open
  and pointwise convergence topologies coincide on $\Omega(S)$ and
  $\Omega(S)$ is a compact semigroup.
\end{Thm}

Recall that a \emph{profinite semigroup} is a residually finite
compact semigroup. Equivalently, it is a compact zero-dimensional
semigroup \cite{Numakura:1957}. 
Since zero-dimensionality is preserved
by product \cite[Theorem~6.2.14]{Engelking:1989} and inherited by
subspaces, we obtain the following result.
\begin{Cor}
  \label{c:TH-profinite}
  If $S$ is a profinite reductive semigroup, then $\Omega(S)$~is a
  profinite semigroup.\qed
\end{Cor}

Given a locally compact closed ideal $I$ of a topological semigroup
$S$, the action of $S$ both on the left and on the right of $I$
determines a homomorphism $\omega^I=(\lambda^I,\rho^I):S\to\Omega(I)$,
which we call the \emph{$I$-representation} of $S$. Where convenient,
we may sometimes write $\omega^I(s)=(\lambda^I(s),\rho^I(s))$ instead
of $\omega^I_s=(\lambda^I_s,\rho^I_s)$ to denote the image of $s\in
S$ under the $I$-representation of~$S$.

\begin{Prop}
  \label{p:TH-continuity}
  For a topological semigroup $S$ and a locally compact closed ideal
  $I$, the $I$-representation of $S$ is continuous.
\end{Prop}

\begin{proof}
  Let $(s_j)_j$ be a convergent net in~$S$ with limit $s$. For every
  convergent net $(u_k)_k$ in~$I$ with limit $u$, the net
  $(s_ju_k)_{j,k}$ converges to~$su$ in~$I$ and so the net
  $\bigl(\lambda^I_{s_j}(u_k)\bigr)_{j,k}$ converges
  to~$\lambda^I_s(u)$ in~$I$. Similarly, the net
  $\bigl((u_k)\rho^I_{s_j}\bigr)_{j,k}$ converges to~$(u)\rho^I_s$. By
  Lemma~\ref{l:continuous-convergence-in-locally-compact}, the net
  $(\omega^I_{s_j})_j=\bigl((\lambda^I_{s_j},\rho^I_{s_j})\bigr)_j$
  converges in~$\Omega(I)$ to $\omega^I_s=(\lambda^I_s,\rho^I_s)$.
  Hence, the function $\omega^I$ is continuous.
\end{proof}

\subsection{Actions on the minimum ideal}
\label{sec:actions}

Let $S$ be a profinite semigroup with a minimum ideal $K$.
Since $K$ is generated by any of its elements, it is a closed ideal.
By Proposition~\ref{p:TH-continuity}, the $K$-representation
$\omega^K=(\lambda^K,\rho^K):S\to\Omega(K)$ of~$S$ is a continuous
homomorphism. Note, that the restriction of $\omega^K$ to the ideal $K$ is faithful,
because $K$ is a completely simple semigroup.

Following
\cite[Definition~4.6.21]{Rhodes&Steinberg:2009qt}, we say that $S$~is
\emph{left mapping (LM)} if the representation $\lambda^K:S\to\Cl
T_K^\ell$ is faithful.
The definition of \emph{right mapping (RM)}
profinite semigroup is dual. If $S$~is both left and right mapping,
then $S$~is said to be \emph{generalized group mapping (GGM)}. A GGM
profinite semigroup whose minimum ideal is not aperiodic is also said
to be \emph{group mapping (GM)}, but we will not be doing this
distinction in this paper. We will also be interested in a weakening
of the GGM property which is easier to prove and powerful enough for
some applications. We say that $S$~is \emph{weakly generalized group
  mapping (WGGM)} if, for all distinct elements $u,v\in S$, either
$\lambda^K(u)\ne\lambda^K(v)$ and $(u)\rho^K\ne(v)\rho^K$, or both $u$
and $v$ belong to~$K$
(and, therefore, $\omega^K(u)\ne\omega^K(v)$).

Since the minimum ideal of a GGM profinite semigroup is a profinite
reductive semigroup, taking into account the results of
Subsection~\ref{sec:translational-hull}, we obtain the following
statement.

\begin{Thm}
  \label{t:TH-wggm}
  Let $S$ be a profinite semigroup with minimum ideal~$K$.
  \begin{enumerate}[(a)]
  \item\label{item:TH-wggm-1} If $S$ is WGGM then
    $\omega^K:S\to\Omega(K)$ is an embedding of topological
    semigroups.
  \item\label{item:TH-wggm-2} If $S$ is GGM then $\Omega(K)$ is a
    profinite semigroup.\qed
  \end{enumerate}
\end{Thm}

Thus, if $S$~is GGM then $\Omega(K)$ is a profinite semigroup in which
$S$ embeds as a closed subsemigroup. There is also an embedding of $S$
in~$\Omega(K)$ under the assumption that $S$ is~WGGM, but then there
is no longer any guarantee that the topological semigroup
$\Omega(K)$~is profinite, and it may not even be a compact semigroup
as an example in Subsection~\ref{sec:TH-compact-simple} shows.

A pseudovariety of semigroups \pv V is \emph{GGM} (respectively
\emph{WGGM}) if, for every finite non-singleton set~$A$, the semigroup
\Om AV~is a GGM (respectively WGGM) semigroup. Trivially, every
pseudovariety of groups is~GGM. We also say that a pseudovariety of
semigroups \pv V is \emph{almost GGM} (respectively \emph{almost
  WGGM}) if, there are arbitrarily large finite alphabets $A$ such
that \Om AV~is GGM (respectively WGGM).

\subsection{The translational hull of a profinite completely simple
  semigroup}
\label{sec:TH-compact-simple}
Since the minimum ideal of a profinite semigroup is a profinite
completely simple semigroup, Theorem~\ref{t:TH-wggm} motivates the
study of the translational hull of profinite completely simple
semigroups. This subsection presents  some preliminary
observations.

We say that a completely simple semigroup has \emph{torsion} if it is
not a rectangular group. By a \emph{$2\times2$ maximal subsemigroup}
of a semigroup $S$ we mean a completely simple subsemigroup with
exactly two \Cl R-classes and two \Cl L-classes which is the union of
\Cl H-classes of~$S$. We say that a
completely simple semigroup $S$ has \emph{full torsion} if it is not a
single \Cl R-class nor a single \Cl L-class and every $2\times2$
maximal subsemigroup has torsion. Note that this condition is
equivalent to each of the following properties, where $e$ and $f$ are
arbitrary idempotents of~$S$:
\begin{itemize}
\item if $e$ and $f$ are neither \Cl R nor \Cl L-equivalent, then the
  product $ef$~is not idempotent;
\item if $ef$ is idempotent, then $ef\in\{e,f\}$.
\end{itemize}
A weaker notion is the following. We say that a completely simple
semigroup $S$ has \emph{plenty of torsion on the left} if, for every
pair of distinct \Cl R-equivalent idempotents $e$ and $f$, there is an
idempotent $g$ from the \Cl L-class of~$e$ such that $fg\ne e$. Note
that, if $S$ has full torsion, then every idempotent $g\ne e$ in the
\Cl L-class of~$e$ has that property.

It is well known that, for an element $s$ of a compact semigroup, the
closed subsemigroup generated by~$s$ contains a unique idempotent,
which we denote~$s^0$. Note that, by definition, it is the limit of
some net of (finite) powers of~$s$. We also denote by $s^{-1}$ the
inverse of~$ss^0$
in the maximal subgroup of~$S$ containing~$ss^0$. 
Thus, we
have $s^0=ss^{-1}=s^{-1}s$. In a profinite semigroup, the traditional
notation is $s^\omega$ instead of~$s^0$ and $s^{\omega-1}$ instead
of~$s^{-1}$, coming from he fact that the $\omega$-power was first
used in the theory of finite semigroups to represent the $n!$-powers
for sufficiently large~$n$.

\begin{Prop}
  \label{p:reduction}
  Let $S$ be a compact completely simple semigroup. Then the following
  hold:
  \begin{enumerate}[(a)]
  \item\label{item:reduction-a} if $u,v\in S$ and
    $\lambda_u=\lambda_v$, then $\lambda_{u^0}=\lambda_{v^0}$ and $u$
    and $v$~are \Cl R-equivalent;
  \item\label{item:reduction-b} if $u,v\in S$ are such that $u^0=v^0$
    and $\lambda_u=\lambda_v$, then $u=v$;
  \item\label{item:reduction-1} the semigroup $S$ is weakly reductive;
  \item\label{item:reduction-2} the canonical mapping
    $S\to\Cl{T}_S^\ell$ is injective if and only if $S$~has plenty of
    torsion on the left.
  \end{enumerate}
\end{Prop}

\begin{proof}
  \eqref{item:reduction-a} Let $u$ and $v$ be two elements of~$S$ and
  suppose that $\lambda_u=\lambda_v$. We deduce that $u^nw=v^nw$ for
  every $w\in S$ and every positive integer~$n$. Hence, we have
  $\lambda_{u^0}=\lambda_{v^0}$. Since $u=uu^0=vu^0$ and
  $v=vv^0=uv^0$, it follows that $u$ and $v$ are \Cl R-equivalent.
  
  \eqref{item:reduction-b} Under the assumptions, we have
  $u=uu^0=uv^0=vv^0=v$.

  \eqref{item:reduction-1} Let $u$ and $v$ be two elements of~$S$ and
  suppose that $(\lambda_u,\rho_u)=(\lambda_v,\rho_v)$. By
  \eqref{item:reduction-a} and its dual, $u$ and $v$ lie in the same
  maximal subgroup of~$S$, that is $u^0=v^0$.
  By~\eqref{item:reduction-b}, it follows that $u=v$.

  \eqref{item:reduction-2} Suppose first that $S$ has plenty of
  torsion on the left. Let $u$ and $v$ be two elements of~$S$. Suppose
  that $\lambda_u=\lambda_v$. We claim that $u=v$.
  By~\eqref{item:reduction-a}, $u^0$ and $v^0$ are \Cl R-equivalent
  idempotents. If $u^0\ne v^0$ then, since $S$~has plenty of torsion,
  there is an idempotent $g$ in the \Cl L-class of~$u^0$ such that
  $v^0g\ne u^0=u^0g$, which contradicts the equality
  $\lambda_{u^0}=\lambda_{v^0}$ given by~\eqref{item:reduction-a}.
  Hence, the equality $u^0=v^0$ holds and so $u=v$
  by~\eqref{item:reduction-b}, which proves the claim.

  Conversely, assume that the canonical mapping is injective and
  suppose that $e$ and $f$ are two distinct \Cl R-equivalent
  idempotents. Then there is some $w\in S$ such that $ew\ne fw$. By
  Green's Lemma, it follows that $ewe\ne fwe$. Hence, we may assume
  that $w\mathrel{\Cl L}e$. Let $t$ be the inverse of~$ew$ in the
  maximal subgroup $H$ containing~$e$. Since $ew$ and $fw$ are
  distinct elements of~$H$, so are $e=ewt$ and $fwt$. Since
  $wt\mathrel{\Cl H}w$ by Green's Lemma, it suffices to observe that
  $wt$ is idempotent. Indeed, $wt\cdot wt=wte\cdot wt=wte=wt$.
\end{proof}

In particular, a profinite completely simple semigroup $S$ is
reductive if and only if it has plenty of torsion both on the left and
on the right. By Corollary~\ref{c:TH-profinite}, $\Omega(S)$~is then a
profinite semigroup.

In the case of a finite discrete completely 0-simple semigroup, the
structure of the translational hull has been described in
\cite[Chapter~7, Facts 2.14 and~2.15]{Krohn&Rhodes&Tilson:1968}. In
view of the translational representation results of
Sections~\ref{sec:wggm-sufficient-conditions}--\ref{sec:CR}, it seems
worthwhile to carry such results to the case of profinite completely
simple semigroups. The analogue of \cite[Chapter~7,
Fact~2.14]{Krohn&Rhodes&Tilson:1968} is the following result, which
only adds topological considerations. The topology we consider on a
Rees matrix semigroup $\Cl M(A,G,B;P)$ is the product tolopogy on $A\times
G\times B$.

\begin{Prop}
  \label{p:TH-CS}
  Let $S=\Cl M(A,G,B;P)$ be a Rees matrix semigroup where $A$ and $B$
  are compact zero-dimensional spaces, $G$~is a profinite group, and
  $P:B\times A\to G$ is a continuous function.
  \begin{enumerate}[(a)]
  \item\label{item:TH-CS-1} The left translations of~$S$ are the
    functions of the form $\lambda(a,g,b)=(\varphi(a),\mu(a)g,b)$,
    where $\varphi\in\Cl T_A^\ell$ and $\mu:A\to G$ is a continuous
    function.
  \item\label{item:TH-CS-2} The right translations of~$S$ are the
    functions of the form $(a,g,b)\rho=(a,g(b)\nu,(b)\psi)$, where
    $\psi\in\Cl T_B^r$ and $\nu:B\to G$
    is a continuous function.
  \item\label{item:TH-CS-3} If $\lambda$ is a left translation of~$S$
    given by $(\varphi,\mu)$ and $\rho$~is a right translation of~$S$
    given by $(\psi,\nu)$, then the pair $(\lambda,\rho)$ is linked if
    and only if the following equation holds for all $a\in A$ and
    $b\in B$:
    \begin{equation}
      \label{eq:link}
      (b)\nu\, P\bigl((b)\psi,a\bigr) %
      =P\bigl(b,\varphi(a)\bigr)\,\mu(a).
    \end{equation}
  \end{enumerate}
\end{Prop}

See also~\cite[Section~5.5.1]{Rhodes&Steinberg:2009qt} for the
connection with linear representations.

An extreme non-reductive case, nevertheless of interest, is that of a
rectangular band $S=A\times B$, where $A$ and $B$ are compact,
respectively left-zero and right-zero semigroups. It follows from
Proposition~\ref{p:TH-CS} that $\Omega(S)$ is isomorphic to the
product $\Cl T_A^\ell\times\Cl T_B^r$. Suppose for instance that $A$
is the usual realization of the Cantor set in the real line. Noting
that the intersections of $A$ with all intervals of the forms $[0,c]$
and $[c,1]$ ($c\in[0,1]\setminus A$) are open, if $(c_n)_n$ is a
sequence in $[0,1]\setminus A$ converging to~1, then the sequence of
characteristic functions of the subsets $[c_n,1]\cap A$ of~$A$, which
belong to $\Cl T_A^\ell$, converges pointwise to the characteristic
function of the subset $\{1\}$ of~$A$, which does not belong to~$\Cl
T_A^\ell$. Since convergence in the compact-open topology implies
pointwise convergence, it follows that the topological semigroup $\Cl
T_A^\ell$~is not compact, whence neither is $\Omega(S)$. The argument
can be easily extended to the case where $A$ or $B$ contains a
subspace homeomorphic to the Cantor set.

In particular, the previous paragraph shows that the assumption that
the semigroup $S$ is reductive cannot be dropped in the statement of
Corollary~\ref{c:TH-profinite}. Nevertheless, by
Proposition~\ref{p:TH-continuity}, a profinite rectangular band
$S$~embeds in~$\Omega(S)$ via the $S$-representation $\omega^S$ by
inner bitranslations and it follows from the results of
Section~\ref{sec:wggm-sufficient-conditions} that $\Omega(S)$ may
admit some much larger profinite subsemigroups containing
$\omega^S(S)$.

\section{Some sufficient conditions for WGGM}
\label{sec:wggm-sufficient-conditions}

In this section, we give some first examples of sufficient conditions
for a pseudovariety to be WGGM. Further examples of WGGM
pseudovarieties are given in Sections~\ref{sec:DS} and~\ref{sec:CR}.

The next result gives somewhat mild conditions under which the
elements of a relatively free profinite semigroup which do not belong
to the minimum ideal $K$ act faithfully on (the left of)~$K$.

Suppose that the subsemigroup of~\Om AV generated by $A$ is freely
generated by~$A$. As has already been observed in
Subsection~\ref{sec:letter-cancelation}, a simple sufficient condition
for this property to hold is that \pv V contain~\pv N. We then
identify the subsemigroup of~\Om AV generated by~$A$ with~$A^+$ and
call its elements \emph{finite words}; all other elements of~\Om AV
are said to be \emph{infinite}. For each $w\in\Om AV$, denote by
$F(w)$ the set of finite words that are factors of~$w$, which we also
call the \emph{finite factors} of~$w$.

Here and in the remainder of the paper, we will also use without
further comment the property that, for a pseudovariety \pv V
containing \pv{LSl}, the finite factors of a product $xy$, with
$x,y\in\Om AV$ are the finite factors of $x$, together with the finite
factors of~$y$, together with the words of the form $x'y'$, where $x'$
is a finite suffix of~$x$ and $y'$~is a finite prefix of~$y$
\cite[Lemma~8.2]{Almeida&Volkov:2006}.

\begin{Prop}
  \label{p:lm}
  Let $A$ be a non-singleton finite set and let \pv V be a monoidal
  pseudovariety of semigroups satisfying the following conditions:
  \begin{enumerate}[(i)]
  \item\label{item:lm-1} $\pv V*\pv D=\pv V$;
  \item\label{item:lm-2} the semigroup \Om AV has content, $0$, and
    $\bar0$ functions.
  \end{enumerate}
  Let $K$ be the minimum ideal of~$\Om AV$. If $u,v\in\Om AV$ are such
  that $\lambda^K(u)=\lambda^K(v)$, then either $u$ and $v$ are equal
  or they both belong to~$K$.
\end{Prop}

\begin{proof}
  Condition~\eqref{item:lm-1} implies, in particular, that \pv V
  contains the pseudovariety~\pv N, so that the free semigroup $A^+$
  can be viewed as a subsemigroup of~\Om AV, namely as the
  subsemigroup generated by~$A$. Since $A^+$ is dense in~\Om AV, a
  necessary and sufficient condition for an element $w$ of~\Om AV to
  belong to~$K$ is that $F(w)=A^+$.
 
  Suppose that $\lambda^K(u)=\lambda^K(v)$ with $u\ne v$. If $u,v\in
  A^+$ then, for an arbitrary $w\in K$, from the equality
  $\lambda^K(u)=\lambda^K(v)$ we obtain $uw=vw$ and thus, in view of
  the hypothesis~\eqref{item:lm-1} and
  Theorem~\ref{t:word-problem-for-V*Dn}, one of $u$ and $v$ must be a
  proper prefix of the other, say $v=uav'$ for some letter $a\in A$
  and some $v'\in A^*$. Then, for $b\in A\setminus\{a\}$, $ub$ is not
  a prefix of~$v$ and so we have $ubw\ne vbw$, which contradicts the
  assumption that $\lambda^K(u)=\lambda^K(v)$. Hence, at least one of
  pseudowords $u$ and $v$ is infinite.

  We claim that $F(u)=F(v)$. Suppose that there is a finite word $s$
  that is a factor of~$u$ but not of~$v$. Let $n=|s|$. Since the
  alphabet $A$~is not a singleton, we may choose a letter %
  $a\in A\setminus\{\mathrm{t}_1(s)\}$. There is some letter $b\in A$
  such that the word $sb$ occurs in the pseudoword $ua^n$. Now, choose
  a letter $c\in A\setminus\{b\}$, which again only requires the
  assumption that $A$~is not a singleton. For an arbitrary element $w$
  of~$K$, as $ua^nscw=va^nscw$, applying $\Phi_n^\pv V$ we obtain %
  $\Phi_n^\pv V(ua^nscw)=\Phi_n^\pv V(va^nscw)$. Since $s$~is not a
  factor of~$v$, the first occurrence of the ``letter'' $sc$
  in~$\Phi_n^\pv V(va^nscw)$ must occur in the factor $\Phi_n^\pv
  V(\mathrm{t}_n(v)a^nsc)$. Moreover, note that the only occurrence of~$s$ as a
  factor of~$va^ns$ is as its suffix. Indeed, we know that it does not
  occur in~$va^n$ because it does not occur in~$v$ and the last letter
  of~$s$ is not~$a$. Thus, all its occurrences must be found
  in~$a^ns$. However, if there is a factorization with $a^ns=xsy$ with
  the word $y$~nonempty, then the number of occurrences of the letter
  $\mathrm{t}_1(s)$ in $a^ns$ is the number of its occurrences in~$s$,
  whereas in $xsy$ it occurs at least that number plus one, as it
  occurs in~$y$. In particular, we conclude that the ``letter'' $sb$
  is not a factor of the prefix of~$\Phi_n^\pv V(va^nscw)$ preceding
  the first occurrence of the ``letter'' $sc$, while the corresponding
  property fails for~$\Phi_n^\pv V(ua^nscw)$, which contradicts the
  hypothesis~\eqref{item:lm-2} in view of Theorem~\ref{t:Phi-vs-Green}
  and the hypothesis~\eqref{item:lm-1}. Hence $F(u)=F(v)$.

  Next, suppose that $s\in A^+\setminus F(u)$. Choose again $a\in
  A\setminus\{\mathrm{t}_1(s)\}$ and let $n=|s|$. For an arbitrary
  $w\in K$, the equality $\Phi_{n-1}^\pv V(ua^nsw)=\Phi_{n-1}^\pv
  V(va^nsw)$ holds. Since, as in the preceding paragraph, the first
  occurrence of the factor $s$ on $ua^ns$ and $va^ns$ is found
  precisely in the suffix position, by the
  hypothesis~\eqref{item:lm-2} we deduce that $\Phi_{n-1}^\pv
  V(ua^ns)=\Phi_{n-1}^\pv V(va^ns)$. In view of the injectivity of the
  function $\Phi_{n-1}^\pv V$ on the set %
  $\Om AV\setminus A_{\le n-1}$, given by
  Theorem~\ref{t:Phi-vs-Green}, we deduce that $ua^ns=va^ns$. By
  Proposition~\ref{p:cancellation}, it follows that $u=v$, in
  contradiction with our initial assumption. This shows that
  $F(u)=F(v)=A^+$ and, therefore, that $u$ and $v$ belong to~$K$,
  which establishes the proposition.
\end{proof}

Combining Proposition~\ref{p:lm} with Proposition~\ref{p:reduction},
respectively parts \eqref{item:reduction-1}
and~\eqref{item:reduction-2}, we obtain the following results.

\begin{Thm}
  \label{t:wggm}
  Let $A$ be a non-singleton finite set and let \pv V be a monoidal
  pseudovariety of semigroups satisfying the following conditions:
  \begin{enumerate}[(i)]
  \item\label{item:wggm-1} $\pv V*\pv D=\pv V$;
  \item\label{item:wggm-2} the semigroup \Om AV has content,
    $0$, $\bar0$, $1$, and $\bar1$ functions.
  \end{enumerate}
  Then the semigroup \Om AV is WGGM.\qed
\end{Thm}

\begin{Thm}
  \label{t:lm}
  Let $A$ be a non-singleton finite set and let \pv V be a monoidal
  pseudovariety of semigroups satisfying the following conditions:
  \begin{enumerate}[(i)]
  \item\label{item:lm-1b} $\pv V*\pv D=\pv V$;
  \item\label{item:lm-2b} the semigroup \Om AV has content, $0$, and
    $\bar0$ functions;
  \item\label{item:lm-3} the minimum ideal of\/~\Om AV has plenty of
    torsion on the left.
  \end{enumerate}
  Then the semigroup \Om AV is~LM.\qed
\end{Thm}

Note that a pseudovariety \pv V satisfying conditions
\eqref{item:lm-1} and~\eqref{item:lm-2} of Proposition~\ref{p:lm} must
contain the pseudovariety \pv K by
Theorem~\ref{t:word-problem-for-V*Dn}, and therefore also $\pv{LI}=\pv
K\vee\pv D$ (see, for instance,
\cite[Corollary~6.4.14]{Almeida:1994a}). In particular, for a
non-singleton finite set $A$, the minimum ideal $K$ of~\Om AV must be
a completely simple semigroup with uncountably many \Cl R and \Cl
L-classes. Since, as it is well known, the restriction of the natural
continuous homomorphism $\Om AV\to\Om A{}(\pv V\cap\pv G)$ to every
maximal subgroup of~$K$ is onto
\cite[Lemma~4.6.10]{Rhodes&Steinberg:2009qt}, the semigroup $K$~has
only trivial subgroups if and only if $\pv V\subseteq\pv A$. In case
$\pv V\subseteq\pv A$, the pseudovariety \pv V can only be GGM if it
is trivial.

In view of Proposition~\ref{p:Cn-properties},
we may apply Theorem~\ref{t:wggm} to obtain the following family of
examples of WGGM pseudovarieties. Except for the case of the
pseudovariety $\pv A=\bar{\pv I}$, this will be improved in
Section~\ref{sec:torsion}.

\begin{Cor}
  \label{c:barH-wggm}
  For every pseudovariety of groups \pv H, the pseudovariety %
  $\bar{\pv H}$~is WGGM.\qed
\end{Cor}

Further examples of WGGM pseudovarieties can be obtained by combining
Theorem~\ref{t:wggm} with Corollary~\ref{c:DAm.VstarA-properties}.

\begin{Cor}
  \label{c:DAm.VstarA-wggm}
  If \pv V is a monoidal pseudovariety of semigroups, then the
  pseudovariety $\pv{DA}\malcev(\pv V*\pv A)$ is WGGM.\qed
\end{Cor}

Both Corollary~\ref{c:barH-wggm} and the following result may be
viewed as particular cases of Corollary~\ref{c:DAm.VstarA-wggm}.

\begin{Cor}
  \label{c:wggm-Cn}
  For every pseudovariety of groups \pv H, the pseudovarieties $\pv
  C_n\cap\bar{\pv H}$~are WGGM.\qed
\end{Cor}

The families of examples in Corollaries~\ref{c:barH-wggm}
and~\ref{c:wggm-Cn}, along with many other examples, can also be
obtained by applying the next theorem.

\begin{Thm}
  \label{t:wggm-equidiv}
  Let \pv V be an equidivisible pseudovariety containing \pv{LSl}.
  Then \pv V~is WGGM.
\end{Thm}

\begin{proof}
  The proof is similar to that of Proposition~\ref{p:lm}. Let $A$~be a
  non-singleton alphabet. We consider distinct elements $u$ and~$v$
  of~\Om AV, not both in the minimum ideal $K$, and assume that
  $\lambda(u)=\lambda(v)$. Since membership in~$K$ is characterized by
  having all finite words as factors, there is some word $s\in A^+$
  that is not a factor of at least one of $u$ and~$v$. Without loss of
  generality, we may as well assume that $s\notin F(v)$. Moreover,
  since every word containing $s$ as factor also has the same
  property, we may replace $s$ by $bsba^{|s|+2}$, where $a$ and $b$
  are distinct letters from~$A$, thereby guaranteeing the additional
  property that $s$~has no nontrivial overlap with itself. For the
  remainder of the proof, $w$ denotes an arbitrary element of~$K$.

  Suppose first that $s$ is also not a factor of~$u$. Since
  $\lambda(u)=\lambda(v)$, we deduce that $usw=vsw$. By
  equidivisibility, the $s$ on the left must match that on the right,
  so that $u=v$, in contradiction with the initial assumption. Hence,
  $s\in F(u)$ and we may assume that $u\in K$.

  Since \pv V contains~\pv{LSl}, the ideal $(\Om AV)^1s(\Om
  AV)^1=\overline{A^*sA^*}$ is a clopen subset of~\Om AV. Taking also
  into account that $(\Om AV)^1$~is compact, it follows that there are
  convergent sequences of words $(x_n)_n$ and $(y_n)_n$ such that
  $u=\lim x_nsy_n$ and $s$~is not a factor of~$x_n$. Let $x=\lim x_n$
  and $y=\lim y_n$. Since $u\in K$, we must have $y\ne1$. Choose $c\in
  A\setminus\{\mathrm{i}_1(y)\}$. From the equality
  $\lambda(u)=\lambda(v)$, we obtain $xsyscw=uscw=vscw$. By
  equidivisibility, the first indicated occurrences of $s$ in $xsyscw$
  and $vscw$ must match, and thus $c$ should be the first letter
  of~$y$, which it is not. This contradiction completes the proof of
  the theorem.
\end{proof}

Taking into account the discussion in
Subsection~\ref{sec:equidivisibility}, we deduce the following result.

\begin{Cor}
  \label{c:wggm-AmV}
  Every pseudovariety closed under concatenation is WGGM.\qed
\end{Cor}

Since the pseudovarieties of the form $\bar{\pv H}$ and $\pv C_n$ are
closed under concatenation, Corollaries \ref{c:barH-wggm}
and~\ref{c:wggm-Cn} may also be obtained as particular cases of
Corollary~\ref{c:wggm-AmV}.

\section{Torsion}
\label{sec:torsion}

Our goal is to prove that certain important pseudovarieites of
semigroups are GGM. For this purpose, we want to apply
Theorem~\ref{t:lm}. While property~\eqref{item:lm-1b} of
Theorem~\ref{t:lm} is already formulated in terms of closure
conditions on the pseudovariety \pv V, properties~\eqref{item:lm-2b}
and~\eqref{item:lm-3} are structural properties of the semigroup \Om
AV, which renders the application of Theorem~\ref{t:lm} difficult. For
property~\eqref{item:lm-2b}, this is not so serious, since we have
already indicated mild conditions that imply it. We proceed to
establish sufficient conditions for property~\eqref{item:lm-3} to
hold, which will allow us to show that many pseudovarieties of
interest satisfy it.

A basic tool to achieve our aim is the semigroup construction
presented in Subsection~\ref{sec:construction}, which in the school of
John Rhodes is known as the \emph{synthesis} construction, which is
used as a tool to build arbitrary (finite or infinite) semigroups
essentially from groups \cite{Rhodes&Allen:1973,Rhodes:1986}. The name
refers to the synthesis of the Rees matrix semigroup construction with
the Krohn-Rhodes Prime Decomposition Theory. In the \emph{synthesis}
theory, the top component in our construction (the semigroup $S$) is
taken to be a group and the other component (the semigroup $T$) grows
successively by the iteration of the construction. In the present
paper, it is rather that other component which plays a special role,
being taken from an atom in the lattice of pseudovarieties of
semigroups, while the top component may be chosen arbitrarily in the
pseudovariety. This construction has recently been used
in~\cite{Almeida&Klima:2011a}, also in connection with irreducibility
properties, and in~\cite{Diekert&Kufleitner&Weil:2011} in a rather
different context.

\subsection{A semigroup construction}
\label{sec:construction}

We follow closely the introduction of the construction given
in~\cite{Almeida&Klima:2011a}. Let $S$ and $T$ be semigroups and let
$f:S^1\to T^1$ be an arbitrary function. The set
$$M(S,T,f)=S\uplus S^1\times T^1\times S^1$$
is a semigroup for the multiplication defined by the following
formulas for all $s,s'\in S$, $s_i,s'_i\in S^1$, $t,t'\in T^1$:
\begin{align*}
  s\cdot s' &= ss' \\
  s\cdot (s_1,t,s_2) &= (ss_1,t,s_2) \\
  (s_1,t,s_2)\cdot s &= (s_1,t,s_2s) \\
  (s_1,t,s_2)\cdot (s'_1,t',s'_2) &= (s_1,tf(s_2s'_1)t',s'_2).
\end{align*}

Given two pseudovarieties of semigroups \pv U and \pv V, we denote by
$\pv U\bullet\pv V$ the pseudovariety generated by all semigroups of
the form $M(S,T,f)$, with $S\in\pv U$ and $T\in\pv V$. As $S$ is a
subsemigroup of $M(S,T,f)$, \pv U is contained in $\pv U\bullet\pv V$.
On the other hand, taking $S=\{1\}$ and $f(1)=1\in T^1$, we obtain a
semigroup $M(S,T,f)$ whose subsemigroup $S^1\times T\times S^1$ is
isomorphic with~$T$, whence \pv V is also contained in $\pv
U\bullet\pv V$.

As observed in \cite[Lemma~3.1]{Almeida&Klima:2011a}, all subgroups
of~$M(S,T,f)$ are isomorphic to subgroups of either $S$ or~$T$.
The following is an immediate corollary of this observation.

\begin{Cor}
  \label{c:closure-under-Rees-matrix-extensions}
  The equation $\bar{\pv H}\bullet\bar{\pv H}=\bar{\pv H}$ holds for
  every pseudovariety of groups \pv H.\qed
\end{Cor}

Here are a few other simple yet useful observations.

\begin{Prop}
  \label{p:bullet-group}
  Let \pv H be a pseudovariety of groups. If\/ \pv V is a
  pseudovariety contained in $\pv{DS}\cap\bar{\pv H}$, respectively
  $\pv{CR}\cap\bar{\pv H}$, then so is~$\pv V\bullet\pv H$.
\end{Prop}

\begin{proof}
  When $G\in\pv H$ and $f:S^1\to G$, the construction $M(S,G,f)$ gives
  a semigroup which is the disjoint union of~$S$ with a completely
  simple semigroup with maximal subgroups isomorphic to~$G$.
\end{proof}

Consider the Rees matrix semigroup %
$K_p=\Cl M(I,\mathbb{Z}/p\mathbb{Z},I, %
\left[
  \begin{smallmatrix}
    0&0\\
    0&1
  \end{smallmatrix}
\right] %
)$, %
where $I$ stands for the set $\{0,1\}$, $p$ is an arbitrary prime, and
we adopt additive notation for the group $\mathbb{Z}/p\mathbb{Z}$.
Note that $K_p$ is generated by the idempotents $(0,0,1)$ and
$(1,0,0)$.

\begin{Lemma}
  \label{l:bullet-torsion-in-CS}
  Suppose that \pv V is a pseudovariety of semigroups satisfying the
  condition $\pv V\bullet\pv{Ab}_p=\pv V$ for a prime~$p$. Then \pv V
  contains the monoid $K_p^1$.
\end{Lemma}

\begin{proof}
  Let $G=\mathbb{Z}/p\mathbb{Z}$. Then, $G$ belongs to $\pv{Ab}_p$ and,
  therefore, also to~\pv V. In case
  $p\ne2$, consider the mapping $g:G\to G$ that sends $2$ to~$1$ and
  every other element to~$0$. Then the subsemigroup
  $\{0\}\cup\{0,1\}\times G\times\{0,1\}$ of~$M(G,G,g)$ is the usual
  representation of the Rees matrix semigroup $K_p$ with an identity
  element adjoined, which shows that $K_p^1\in\pv V$. In case $p=2$,
  we consider the mapping %
  $h:G\times G\to G$ which sends $(1,1)$ to~$1$ and every other
  element to~$0$. Then, it is easily verified that the subsemigroup
  $\{(0,0)\}\cup\{(0,0),(1,0)\}\times G\times\{(0,0),(0,1)\}$
  of~$M(G\times G,G,h)$ is isomorphic with the monoid~$K_2^1$.
\end{proof}

\subsection{Some combinatorial lemmas}
\label{sec:combinatorics}

Before stating and proving the result that accomplishes the
requirement of plenty of torsion in the minimum ideal of
Theorem~\ref{t:lm}, we prove some auxiliary combinatorial results.

\begin{Lemma}
  \label{l:stretch}
  Let $A$ be a non-singleton finite alphabet and let \pv V be a
  pseudovariety containing \pv{LSl}. Let $x,y,z\in\Om AV$ be arbitrary
  pseudowords, and suppose that $s\in A^+$ is a word such that
  $\mathrm{t}_2(s)$ is a square and $s$ is not a factor of~$x$. Then,
  there exists a word~$r\in A^+$ such that $xr\notin\{y,z\}$ and the
  only occurrence of $s$ as a factor of~$xrs$ is as a suffix.
\end{Lemma}

\begin{proof}
  Let $a\in A$ be the letter such that $\mathrm{t}_2(s)=a^2$ and
  choose a letter $b\in A\setminus\{a\}$. Consider the words
  $r_1=b(ab)^k$, $r_2=b(ab)^kb$, and $r_3=b(ab)^kb^2$, where
  $k\ge|s|/2$. As $\pv D\subseteq\pv{LSl}\subseteq\pv V$, the
  pseudowords $xr_i\in\Om AV$ ($i=1,2,3$) are distinct, for so are
  their suffixes of length~3. Hence, at least one of them, say $xr_i$,
  is different from both $y$ and $z$; let $r=r_i$. We claim that the
  only occurrence of $s$ as a factor of~$xrs$ is as a suffix.

  Suppose that there is an occurrence of~$s$ as a factor of~$xrs$
  other than as a suffix. Since $s$~is not a factor of~$x$, any
  occurrence of~$s$ as a factor of~$xrs$ must be obtained as a factor
  of some $us$, where $u$ is a finite suffix of $xr$. We may therefore
  take such $u\ne1$ as short as possible, so that $s$~is a prefix
  of~$us$. In particular, there exists a nonempty word $v$ such that
  the equality $us=sv$ holds in~$A^+$.

  Note that, by construction, the word $r$~ends with the letter~$b$.
  If $v$~is a letter, then it is the last letter of~$s$, namely~$a$.
  Since $us$ and $sv$ have the same number of occurrences of the
  letter~$a$ and $|u|=|v|$, it follows $u=a$, which contradicts the
  fact that $u$~is a suffix of~$xr$. Hence, $v$~has length at least
  two and, therefore, $a^2$~is a suffix of~$v$. Since $us=sv$, the
  words $us$ and $sv$ have the same number of occurrences of the
  factor $a^2$. As $u$~ends with the letter~$b$, it follows that
  $a^2$~is a factor of~$u$. By the choice of~$r$ and as $u$~is a
  suffix of~$xr$, every occurrence of $a^2$ in~$u$ must come from~$x$.
  Thus, there is a factorization $u=x'r$, where $x'$~is a suffix
  of~$x$. Since $|u|\ge|r|>|s|$ and $us=sv$, we conclude that $s$~is a
  prefix of~$u=x'r$. As every occurrence of~$a^2$ in~$u$ comes
  from~$x'$, we deduce that $s$~is actually a prefix of~$x'$, whence a
  factor of~$x$, in contradiction with the hypothesis of the lemma.
  This establishes the claim.
\end{proof}

The next lemma may be viewed as a connectivity property of the de
Bruijn graphs on alphabets with at least two letters.

\begin{Lemma}
  \label{l:deBruijn-connectivity}
  Let $w\in A^*$ be a word and suppose that $a,b\in A$ are two
  distinct letters. Then there is a word $t\in A^*$ such that the word
  $wt$ has only one occurrence of the factor $aw$, namely as a suffix,
  and no occurrence of~$bw$.
\end{Lemma}

\begin{proof}
  By counting the number of occurrences of the letter $b$, we see that
  no word of the form $wa^m$ admits $bw$~as a factor. Then,
  for $m=|w|$, $bw$ is not a factor of $wa^mw$
  but clearly $aw$~is. Hence, the shortest prefix of~$wa^mw$ that
  admits $aw$~as a factor is a word of the required form $wt$.
\end{proof}

\subsection{Torsion accomplished}
\label{sec:torsion-accomplished}

We now come to the announced sufficient closure conditions on a
pseudovariety \pv V for \Om AV to have a minimum ideal with plenty of
torsion, provided $|A|\ge 2$.

\begin{Thm}
  \label{t:plenty-of-torsion}
  Let $A$ be a non-singleton finite set and let \pv V be a monoidal
  pseudovariety of semigroups satisfying the following conditions:
  \begin{enumerate}[(i)]
  \item\label{item:plenty-of-torsion-1} $\pv V*\pv D=\pv V$;
  \item\label{item:plenty-of-torsion-2} the semigroup \Om AV has
    content, $0$ and $\bar0$ functions;
  \item\label{item:plenty-of-torsion-3} $\pv V\bullet\pv{Ab}_p=\pv V$
    for some prime~$p$.
  \end{enumerate}
  Then the minimum ideal of\/~\Om AV has plenty of torsion on the
  left.
\end{Thm}

\begin{proof}
  Let $p$~be a prime verifying
  condition~\eqref{item:plenty-of-torsion-3}. Also, let $K$ be the
  minimum ideal of~\Om AV and let $u$ and $v$ be distinct \Cl R-equivalent
  idempotents of~$K$. We claim that
  $\lambda^K(u)\ne\lambda^K(v)$. In view of
  Proposition~\ref{p:reduction}, this is sufficient to establish the
  theorem.

  \smallskip

  Suppose first that the following condition holds:
  \begin{equation}
    \label{eq:case-1}
    \mathrm{t}_k(u)=\mathrm{t}_k(v)
    \mbox{ for every } k\ge1.
  \end{equation}
  Since $u,v\in K$, we have $F(u)=F(v)=A^+$. Suppose that, for every
  word $s\in A^+\setminus A$, there are factorizations $u=u_ssz_s$ and
  $v=v_ssz_s$ such that $s\notin F(\mathrm{t}_{|s|-1}(s)z_s)$.
  Note
  that the set $A^+$ ordered by Green's relation $\ge_\Cl J$ is upper
  directed. By compactness of the space~$(\Om AV)^1$, the net
  $\bigl(u_s,v_s,s,z_s\bigr)_{s\in A^+}$ admits a convergent subnet,
  say with limit $(u',v',r,z)$. Since multiplication is continuous, we
  deduce the equalities $u=u'rz$ and $v=v'rz$. By construction,
  $F(r)=A^+$, and so $r$~belongs to~$K$, whence so does $rz$. Hence,
  the \Cl R-equivalent idempotents
  $u$ and $v$ are both \Cl L-equivalent to~$rz$, in contradiction with
  the assumption that they are distinct. Thus, there is
  some finite word $s\in A^+$ of length at least~2 such that, for all
  factorizations $u=u_1su_2$ and $v=v_1sv_2$ where $s$ is a factor
  of neither $\mathrm{t}_n(s)u_2$ nor $\mathrm{t}_n(s)v_2$, with
  $n=|s|-1$, the pseudowords $u_2$ and $v_2$ are infinite
  and $\mathrm{t}_n(s)u_2\ne\mathrm{t}_n(s)v_2$, which is equivalent
  to $u_2\ne v_2$ by Proposition~\ref{p:cancellation}. 
  Note that,
  if $s$ has this property, then so does every word of which $s$ is a
  factor. Hence, we may assume without loss of generality that $s$~has
  the form $s=bs'ba^{\ell+2}$ %
  where $a$ and $b$ are distinct letters from~$A$, $s'\in A^+$, and
  $\ell=|s'|$. The advantage of such a choice is that $s$ does not
  overlap with itself, which makes it easier to locate occurrences of
  $s$ in a pseudoword, and $s$~ends with the square of a letter, which
  allows us to invoke Lemma~\ref{l:stretch}. To simplify the notation,
  we let %
  $\bar s=\mathrm{t}_n(s)$ and %
  $\bar u=\mathrm{t}_n(u)$. Also consider the word
  $s_0=\mathrm{i}_n(s)$.

  Write $u=u_0su_3$ and $v=v_0sv_3$, where $s\notin F(u_0)\cup
  F(v_0)$. Since $u$ and $v$ are \Cl R-equivalent elements in~$K$,
  there is $w\in(\Om AV)^1$ such that $v=uw$. By
  Theorem~\ref{t:word-problem-for-V*Dn} and
  condition~\eqref{item:plenty-of-torsion-1}, we know that
  $\Phi_n^{\pv V}(v)=\Phi_n^{\pv V}(uw)$. Since $s$ has no overlap
  with itself and it is
  not a factor of either $u_0$ or $v_0$, by condition
  \eqref{item:plenty-of-torsion-2} the first occurrences of~$s$, as a
  letter from $A_n$, in~$\Phi_n^{\pv V}(u)$ and $\Phi_n^{\pv V}(v)$ from
  left to right as well as the prefixes that determine them must be
  the same. Hence,  we
  must have $\Phi_n^{\pv V}(u_0s_0)=\Phi_n^{\pv V}(v_0s_0)$ and, therefore, also
  $u_0s_0=v_0s_0$.
  Taking into account that $u$ and $v$~are
  idempotents, we obtain the following equalities:
  \begin{equation}
    \label{eq:expression-u,v}
    u = (u_0su_3u_1su_2)^\omega %
    \quad\mbox{and}\quad %
    v = (u_0sv_3v_1sv_2)^\omega.
  \end{equation}
  We are interested in counting, modulo~$p$, occurrences of
  pseudowords of the form $sts$ in~$uw$ and $vw$, where $s$~does not
  occur in~$t$ and $w\in K$ remains to be chosen. The occurrences of
  such factors in the sections $su_3u_1s$ and $sv_3v_1s$ of the
  expressions~\eqref{eq:expression-u,v}, pose no problem because of
  the exponents $\omega$. Since $s$~cannot be found as a factor of any
  of~$u_0,u_2,v_2$, what we have to worry about is the possible
  occurrence of~$s$ as a product $s_1s_2$ with $u_0=s_2u_4$
  $u_2=u_5s_1$, and similarly for~$v$. As $s$~does not overlap with
  itself, there can be at most one such factorization and, in view
  of~\eqref{eq:case-1}, there is one coming from $u$ if and only if
  the similar factorization comes from~$v$. In this case, we have
  factorizations $u_0=s_2u_4$, $u_2=u_5s_1$, and $v_2=v_5s_1$. If such
  a case does not occur, then we take $u_5=u_2u_0$ and $v_5=v_2u_0$.

  By Lemma~\ref{l:stretch}, there exists a finite word $r$ such that
  $\bar su_2r\notin\{\bar su_5,\bar sv_5\}$ and $s$~is not a factor
  of~$\bar su_2r$. Since
  $\mathrm{t}_n(u_2)=\mathrm{t}_n(u)=\mathrm{t}_n(v)=\mathrm{t}_n(v_2)$
  by~\eqref{eq:case-1}, we also know that
  $s$ is not a factor of~$\bar sv_2r$.
  By Proposition~\ref{p:cancellation}, we obtain
  \begin{equation}
    \label{eq:distinction}
    \bar su_2rs_0\notin\{\bar sv_2rs_0,\bar su_5s_0,\bar sv_5s_0\}.
  \end{equation}
  As $\Phi_n^\pv V$~is injective on the set %
  $\Om AV\setminus A_{\le n}$ by Theorem~\ref{t:Phi-vs-Green}, the
  non-membership condition~\eqref{eq:distinction} is preserved after
  applying this function. Hence, as $s$ is not a factor of any of the
  pseudowords in~\eqref{eq:distinction}, there is some semigroup $S$
  from~\pv V and some continuous homomorphism
  $\varphi:\Om{A_{n+1}\setminus\{s\}}V\to S$ such that the following
  condition holds:
  \begin{equation}
    \label{eq:S-distinction}
    \varphi\bigl(\Phi_n^\pv V(\bar su_2rs_0)\bigr) %
    \notin\{
    \varphi\bigl(\Phi_n^\pv V(\bar sv_2rs_0)\bigr), %
    \varphi\bigl(\Phi_n^\pv V(\bar su_5s_0)\bigr), %
    \varphi\bigl(\Phi_n^\pv V(\bar sv_5s_0)\bigr)\}.
  \end{equation}
  Consider the additive group $G=\mathbb{Z}/p\mathbb{Z}$ and the
  semigroup $M(S,G,\xi)$, where $\xi:S^1\to G$ maps
  $\varphi\bigl(\Phi_n^\pv V(\bar su_2rs_0)\bigr)$ to the
  generator~$1$ and every other element to the idempotent~$0$. Since
  $p$~verifies condition~\eqref{item:plenty-of-torsion-3}, the
  semigroup $M(S,G,\xi)$ belongs to~\pv V. We may therefore extend
  $\varphi$ to a continuous homomorphism $\psi:\Om{A_{n+1}}V\to
  M(S,G,\xi)$ by letting
  $$\psi(\alpha)=
  \begin{cases}
    (1,0,1) &\mbox{if } \alpha=s, \\
    \varphi(\alpha) &\mbox{if } \alpha\in A_{n+1}\setminus\{s\}.
  \end{cases}
  $$
  We claim that %
  \begin{equation}
    \psi\bigl(\Phi_n^\pv V(urs)\bigr)\ne%
    \psi\bigl(\Phi_n^\pv V(vrs)\bigr).
    \label{eq:claim-inequality}
  \end{equation}
  Let us first look at the consequences of this claim, postponing its
  proof until the next paragraph. Since the two sides of the
  inequality \eqref{eq:claim-inequality} fall in the same subgroup of
  the minimum ideal of~$M(S,G,\xi)$, Green's Lemma implies that
  $\psi\bigl(\Phi_n^\pv V(ursw)\bigr)\ne%
  \psi\bigl(\Phi_n^\pv V(vrsw)\bigr)$ for every $w\in\Om AV$, in
  particular for $w\in K$. Hence, we have %
  $\Phi_n^\pv V(ursw)\ne\Phi_n^\pv V(vrsw)$, whence $ursw\ne vrsw$
  Thus, the claim yields the inequality
  $\lambda^K(u)\ne\lambda^K(v)$ under the assumption that
  condition~\eqref{eq:case-1} holds. %

  To prove the claim (\ref{eq:claim-inequality}), we use the
  expressions \eqref{eq:expression-u,v} for~$u$ and~$v$. Taking into
  account the definition of~$\psi$ and how the multiplication
  in~$M(S,G,\xi)$ is defined, we may then compute
  \begin{align*}
    \psi\bigl(\Phi_n^\pv V(urs)\bigr) %
    & %
    = \psi\bigl(\Phi_n^\pv V((u_0su_3u_1su_2)^\omega rs)\bigr)
    \\
    & %
    = \psi\Bigl(
     \Phi_n^\pv V(u_0s_0)s %
    \bigl( %
     \Phi_n^\pv V(\bar su_3u_1s_0)s %
     \Phi_n^\pv V(\bar su_2u_0s_0)s %
    \bigr)^{\omega-1} %
    \\
    &\qquad\quad %
    \Phi_n^\pv V(\bar su_3u_1s_0)s %
    \Phi_n^\pv V(\bar su_2rs_0)s %
    \Bigr)
    \\
    &= %
    \bigl(\varphi(\Phi_n^\pv V(u_0s_0)),g+1,1\bigr),
  \end{align*}
  where
  $$g=
  \begin{cases}
    -\xi\bigl(\varphi(\Phi_n^\pv V(\bar su_4s_0))\bigr) %
    &\mbox{ if } s\in F(u_2u_0) \\
    0 %
    &\mbox{ otherwise.}
  \end{cases}
  $$
  Similarly, we obtain %
  $\psi\bigl(\Phi_n^\pv V(vrs)\bigr) %
  =\bigl(\varphi(\Phi_n^\pv V(u_0s_0)),g,1\bigr)$,
  where $g$ is also given by the above formula. Hence, we have
  $\psi\bigl(\Phi_n^\pv V(urs)\bigr)\ne\psi\bigl(\Phi_n^\pv
  V(vrs)\bigr)$, as was claimed.

  \smallskip
  
  It remains to treat the cases where \eqref{eq:case-1} fails. Let $k$
  be minimum such that $\mathrm{t}_k(u)\ne\mathrm{t}_k(v)$ and let
  $n=k-1$, $s=\mathrm{t}_n(u)=\mathrm{t}_n(v)$, $as=\mathrm{t}_k(u)$,
  and $bs=\mathrm{t}_k(v)$. In particular, $a$ and $b$ are distinct
  letters from~$A$. By Lemma~\ref{l:deBruijn-connectivity}, there are
  words $r,t\in A^*$ such that $sr$ has only one occurrence of $as$ as
  a factor, namely as a suffix, and none of~$bs$, and $st$ has only
  one occurrence of $bs$ as a factor, namely as a suffix, and none
  of~$as$.

  By Theorem~\ref{t:Phi-vs-Green}, the elements %
  $u'=\Phi_n^\pv V(u)$ and $v'=\Phi_n^\pv V(v)$ of~$\Om{A_k}V$ are \Cl
  R-equivalent but not \Cl L-equivalent. Since the monoid $K_p^1$
  belongs to~\pv V by Lemma~\ref{l:bullet-torsion-in-CS}, there is a
  continuous homomorphism $\varphi:\Om{A_k}V\to K_p^1$ that maps $as$
  to~$(0,0,1)$, $bs$ to~$(1,0,0)$, and every other element of~$A_k$ to
  the identity element~$1$. Since the pseudowords $u'$ and $v'$ end
  respectively with the letters $as$ and $bs$, their images
  $\varphi(u')$ and $\varphi(v')$ must be \Cl R-equivalent but not \Cl
  L-equivalent. More precisely, there exist $i\in\{0,1\}$ and
  $g,h\in\mathbb{Z}/p\mathbb{Z}$ such that $\varphi(u')=(i,g,1)$ and
  $\varphi(v')=(i,h,0)$. If $g=h$, then we obtain %
  \begin{align*}
    \varphi\bigl(\Phi_n^\pv V(ut)\bigr)& %
    = \varphi\bigl(\Phi_n^\pv V(u)\Phi_n^\pv V(st)\bigr) %
    = (i,g,1)(1,0,0)=(i,g+1,0),\\
    \varphi\bigl(\Phi_n^\pv V(vt)\bigr)&%
    = \varphi\bigl(\Phi_n^\pv V(v)\Phi_n^\pv V(st)\bigr) %
    = (i,g,0)(1,0,0)=(i,g,0).
  \end{align*}
  Similarly, in case $g\ne h$, we may calculate
  $$\varphi\bigl(\Phi_n^\pv V(ur)\bigr)=(i,g,1) %
  \quad\mbox{and}\quad %
  \varphi\bigl(\Phi_n^\pv V(vr)\bigr)=(i,h,1).$$
  In both cases, by Green's Lemma, we may then take any $w\in K$ to
  deduce that either $\varphi\bigl(\Phi_n^\pv
  V(utw)\bigr)\ne\varphi\bigl(\Phi_n^\pv V(vtw)\bigr)$ or %
  $\varphi\bigl(\Phi_n^\pv V(urw)\bigr)\ne\varphi\bigl(\Phi_n^\pv
  V(vrw)\bigr)$. %
  This shows that
  $utw\ne vtw$ or $urw\ne vrw$ and, therefore,
  $\lambda^K(u)\ne\lambda^K(v)$ which concludes the proof of the theorem.
\end{proof}

Combining Theorems~\ref{t:lm} and~\ref{t:plenty-of-torsion}, we obtain
the following result, which allows us to show that many
pseudovarieties of interest are~GGM.

\begin{Thm}
  \label{t:gm}
  Let \pv V be a monoidal pseudovariety of semigroups satisfying the
  following conditions:
  \begin{enumerate}[(i)]
  \item\label{item:gm-1} $\pv V*\pv D=\pv V$;
  \item\label{item:gm-2} the pseudovariety \pv V contains~\pv{Sl} and
    it is closed under Birget expansions;
  \item\label{item:gm-3} $\pv V\bullet\pv{Ab}_p=\pv V$ for some
    prime~$p$.
  \end{enumerate}
  Then the pseudovariety \pv V is~GGM.\qed
\end{Thm}

Taking into account
Corollary~\ref{c:closure-under-Rees-matrix-extensions}, we may apply
Theorem~\ref{t:gm} to many familiar pseudovarieties, thus improving
Corollary~\ref{c:barH-wggm}.

\begin{Cor}
  \label{c:gm-barH}
  For every nontrivial pseudovariety of groups \pv H, the
  pseudovariety $\bar{\pv H}$ is~GGM.\qed
\end{Cor}

\section{WGGM for subpseudovarieties of
  \texorpdfstring{\pv{DS}}{DS}}
\label{sec:DS}

Many pseudovarieties of interest are contained in the pseudovariety
\pv{DS}. Although \pv{DS} can be easily seen to be closed under Birget
expansions and $\pv{DS}\bullet\pv G=\pv{DS}$ by
Proposition~\ref{p:bullet-group}, $\pv{DS}*\pv D\ne\pv{DS}$ since, for
instance, $\pv{Sl}*\pv D$ contains the aperiodic five-element Brandt
semigroup $B_2$ while \pv{DS} is precisely the largest pseudovariety
that does not contain~$B_2$. Thus, to establish that suitable
subpseudovarieties of~\pv{DS} are GGM, we have to develop an
alternative approach. In fact, we only manage to prove WGGM. The basic
idea is that, for every pseudovariety \pv V in the interval $[\pv
J,\pv{DS}]$, where \pv J is the pseudovariety of all finite \Cl
J-trivial semigroups, membership in the minimum ideal of~\Om AV is
characterized by the property of admitting all finite words as
subwords\footnote{ For a pseudoword $w\in\Om AV$ and a finite word
  $s\in A^+$, we say that $s$~is a \emph{subword} of~$w$ if there are
  factorizations $s=s_1\cdots s_n$ and $w=w_0s_1w_1\cdots s_nw_n$, where
$s_1,\dots ,s_n\in A$ and $w_0,\dots ,w_n\in (\Om AV)^1$.}
\cite[Theorems~8.1.7 and~8.1.10]{Almeida:1994a}. Since the
minimum ideal of~\Om AJ is trivial, we need a larger pseudovariety to
be able to start a program mimicking that developed in previous
sections. Since $\pv{Sl}\subseteq\pv J$, such a semigroup \Om AV has
automatically a content function.

In compensation for dropping the hypothesis $\pv V*\pv D=\pv V$, we
need to reinforce the hypothesis of having $0$ and $\bar0$ functions
with the stronger condition of uniqueness of left basic
factorizations.

\begin{Prop}
  \label{p:lm-DS}
  Let \pv V be a pseudovariety in the interval $[\pv J,\pv{DS}]$ and
  let $A$ be a non-singleton finite alphabet. If \pv V has unique left
  basic factorizations and $u,v\in\Om AV$ are such that
  $\lambda^K(u)=\lambda^K(v)$ then either $u$ and $v$ are equal or
  they both belong to the minimum ideal $K$.
\end{Prop}

\begin{proof}
  Suppose that there is some finite subword of~$u$ that is not a
  subword of~$v$. Choose $za$ to be such a word of minimum length,
  with $a\in A$, so that $za$ is a subword of~$u$ but
  not of~$v$, while $z$ is a subword of~$v$. Since $|A|\ge2$, we may
  choose a letter $b\in A$ which is different from the first letter of
  the remainder in the left-greedy occurrence of $za$ as a subword
  in~$ua$.

  Let $w$ be any element of the minimum ideal $K$ of~\Om AV. Since
  $za$ is not a subword of~$v$, but $z$~is, the remainder of the left
  greedy occurrence of $za$ in~$vabw$ is $bw$. On the other hand,
  since $za$ is a subword of~$u$ and by the choice of the letter $b$,
  the remainder of the left greedy occurrence of $za$ in~$uabw$ starts
  with a letter different from~$b$. In view of uniqueness of left
  basic factorizations in~\Om AV, it follows that $uabw\ne vabw$,
  which contradicts the assumption that $\lambda^K(u)=\lambda^K(v)$.
  Hence $u$ and $v$ have the same finite subwords.

  Suppose that there is some finite word that is not a subword
  of~$u$. Let $za$ be such a word of minimum length, with $a\in A$.
  Let $w$ be an arbitrary element of~$K$. From the hypothesis that
  $\lambda^K(u)=\lambda^K(v)$, we deduce that $uaw=vaw$. But, since
  $z$~is a subword of both $u$ and~$v$, while $za$~is not, in the left
  greedy occurrence of~$za$ in~$uaw$ and $vaw$, the indicated
  occurrences of~$a$ must be the chosen occurrences of the last letter
  of~$za$. Hence, we must have $u=v$. Thus, if $u\ne v$, then $u$ and
  $v$ both admit all finite words as subwords and so they belong
  to~$K$.
\end{proof}

Combining Proposition~\ref{p:lm-DS} with its dual and
Proposition~\ref{p:reduction}\eqref{item:reduction-1}, we deduce the
following result.

\begin{Thm}
  \label{t:wggm-DS}
  Let \pv V be a pseudovariety in the interval $[\pv J,\pv{DS}]$ that
  has unique left and right basic factorizations. Then \pv V~is
  WGGM.\qed
\end{Thm}

Combining Theorem~\ref{t:wggm-DS} and
Corollary~\ref{c:unique-lbf-egs}, we obtain the following important
examples of WGGM pseudovarieties.

\begin{Cor}
  \label{c:wggm-DS-egs}
  For an arbitrary pseudovariety of groups \pv H, the pseudovarieties
  $\pv{DO}\cap\bar{\pv H}$ and $\pv{DS}\cap\bar{\pv H}$ are WGGM.\qed
\end{Cor}

We conjecture that every pseudovariety of the form %
$\pv{DS}\cap\bar{\pv H}$, where \pv H is a nontrivial pseudovariety of
groups, is actually GGM.

Another interesting class of examples is obtained by combining
Theorem~\ref{t:wggm-DS} with Corollary~\ref{c:unique-bf-Cn}, which
leads to the following result.

\begin{Cor}
  \label{c:wggm-DSCn}
  The pseudovarieties $\pv{DS}\cap\pv C_n$~are WGGM.\qed
\end{Cor}

\section{GGM for subpseudovarieties of
  \texorpdfstring{\pv{CR}}{CR}}
\label{sec:CR}

This section is dedicated to proving GGM or its weakened versions for
various subpseudovarieties of~\pv{CR}.

\begin{Prop}
  \label{p:wggm-CR}
  Let $A$ be a non-singleton finite alphabet and let \pv V be a
  subpseudovariety of~\pv{CR} such that \Om AV has content, $0$, and
  $\bar0$ functions. Suppose further that at least one of the
  following conditions holds:
  \begin{enumerate}[(i)]
  \item\label{item:wggm-CR-1} $|A|\ge3$ and \pv V contains~\pv{RZ};
  \item\label{item:wggm-CR-2} the pseudovariety \pv V~contains some
    nontrivial group.
  \end{enumerate}
  If $u,v\in\Om AV$ are two distinct elements, then either
  $\lambda^K(u)\ne\lambda^K(v)$ or both $u$ and $v$ belong to the
  minimum ideal.
\end{Prop}

\begin{proof}
  Recall that \pv{CR} is contained in~\pv{DS}. Thus, \pv V belongs to
  the interval $[\pv{Sl},\pv{DS}]$, and so the regular \Cl J-classes
  of~\Om AV are characterized by the content of their
  elements \cite[Theorem~8.1.7]{Almeida:1994a}. Since $\pv
  V\subseteq\pv{CR}$, there are no other \Cl J-classes. In particular,
  the minimum ideal $K$ consists precisely of the elements of full
  content. Suppose that $u,v\in\Om AV$ are distinct elements such that
  $\lambda^K(u)=\lambda^K(v)$. Let $w$ be an arbitrary element of~$K$.

  Suppose that there is some letter $a\in A\setminus(c(u)\cup c(v))$.
  From the equality $uaw=vaw$, applying the function $0$ sufficiently
  many times, we obtain $u=v$, in contradiction with the hypothesis.
  Hence, every letter from~$A$ occurs in either $u$ or~$v$.
  Assume that there is a letter $a$ that occurs in~$u$ but not
  in~$v$. Let $u=u_0au_1$ be a factorization of~$u$ such that $a\notin
  c(u_0)$. Then, from the equality $\lambda^K(u)=\lambda^K(v)$, we
  obtain $u_0au_1aw=uaw=vaw$, which entails $u_0=v$, since $a\notin
  c(v)$, whence $c(v)\subseteq c(u)$ and $c(u)=A$.

  Suppose first that \pv V contains some nontrivial group. Let $p$~be
  a prime such that $\pv{Ab}_p\subseteq\pv V$ and consider the natural
  projection $\pi:\Om AV\to\Om A{Ab}_p$, where %
  $\Om A{Ab}_p\simeq(\mathbb{Z}/p\mathbb{Z})^A$; the mapping $\pi$
  counts modulo $p$ the number of occurrences of each letter. Let $b$
  be a letter from~$c(v)$ and note that $\lambda^K(u)=\lambda^K(v)$
  also yields $u_0au_1baw=ubaw=vbaw$, which now entails $u_0=vb$,
  since $a\notin c(vb)$. Hence $\pi(v)=\pi(u_0)=\pi(vb)$, which is
  absurd since the $b$-components of $\pi(v)$ and $\pi(vb)$ are
  distinct.

  Consider finally the aperiodic case, where %
  $\pv{RZ}\subseteq\pv V\subseteq\pv{CR}\cap\pv A=\pv B$ and
  $|A|\ge3$. Choose $b\in A\setminus\{a,\mathrm{t_1}(u_0)\}$. Then,
  from the equality $\lambda^K(u)=\lambda^K(v)$, we obtain
  $u_0au_1baw=ubaw=vbaw$, whence $u_0=vb$. Hence, we have
  $\mathrm{t}_1(u_0)=b$, in contradiction with the choice of~$b$.
\end{proof}

Note that $\lambda^K(ab)=\lambda^K(a)$ in~$\Om{\{a,b\}}B$, which shows
that the restriction $|A|\ge3$ cannot be dropped from the hypothesis
of Proposition~\ref{p:wggm-CR} in the aperiodic case.

Combining Propositions~\ref{p:wggm-CR}
and~\ref{p:reduction}\eqref{item:reduction-1}, we obtain the following
result.

\begin{Thm}
  \label{t:weak-gm-CR}
  Let \pv V be a pseudovariety in the interval $[\pv{Sl},\pv{CR}]$,
  and suppose \pv V~is closed under Birget expansions. Then \pv V is
  almost WGGM. Moreover, if\/ \pv V contains some nontrivial
  group then \pv V~is WGGM.\qed
\end{Thm}

Combining Theorem~\ref{t:weak-gm-CR} with
Proposition~\ref{p:Cn-properties}, we obtain the following family of
further examples of WGGM pseudovarieties.

\begin{Cor}
  \label{c:wggm-CRCn}
  For $n>0$, the pseudovarieties $\pv{CR}\cap\pv C_n$~are WGGM.\qed
\end{Cor}

Theorem~\ref{t:weak-gm-CR} also yields that \pv B~is almost WGGM
while, for a nontrivial pseudovariety of groups~\pv H,
$\pv{CR}\cap\bar{\pv H}$~is WGGM. The remainder of this section is
dedicated to proving that $\pv{CR}\cap\bar{\pv H}$~is actually GGM,
with one exception, in which it is almost GGM.

The pseudoidentity problem for $\pv{CR}\cap\bar{\pv H}$ has been
solved in~\cite{Almeida&Trotter:2001}. The solution is in a sense
similar to Theorem~\ref{t:word-problem-for-V*Dn} but involving other
parameters. Two pseudowords must have the same content to be equal
over~$\pv{CR}\cap\bar{\pv H}$. The roles of $\mathrm{i}_n$ and
$\mathrm{t}_n$ are played by the pairs of functions $(0,\bar0)$ and
$(1,\bar1)$, while that of the function $\Phi_n$~is taken by the
profinite version of Ka\v dourek and Pol\'ak's characteristic function
\cite{Kadourek&Polak:1986}. For a word $w$, the \emph{characteristic
  sequence} $\chi(w)$ is the word that is obtained by reading from
left to right the maximal factors that miss exactly one letter
from~$w$. For pseudowords, the definition is technically complicated
and is made in terms of a pseudopath in a certain free profinite
category over a profinite graph with infinitely many vertices. In
fact, this poses in general delicate problems which were overlooked
in~\cite{Almeida&Trotter:2001}, as observed
in~\cite{Almeida&ACosta:2007a}, namely the free category generated by
the graph may not be dense in the free profinite category over the
same graph. However, using the techniques
of~\cite{Almeida&ACosta:2007a}, A. Costa and the first author have
been able to show that, due to the special nature of the graph, the
approach in~\cite{Almeida&Trotter:2001} works fine as the density
condition is fulfilled \cite{Almeida&Costa:2015a}.

The graph in question associated with a pseudovariety of groups \pv H,
denoted $\partial_X\pv H$, is similar to the de Bruijn graph, being
associated with a fixed subset $X$ of the alphabet $A$, containing at
least two letters. The edges are the pseudowords $w$ with content
contained in~$X$ and missing just one letter, where two edges are
identified if the pseudoidentity they determine is valid
in~$\pv{CR}\cap\bar{\pv H}$. The extremes of such an edge $w$ are the
pseudowords $0(w)$ and $1(w)$, missing exactly two letters from~$X$.
Let $[X]$ be the set of words in $A^+$ of content~$X$. The
characteristic sequence can be viewed as a function from $[X]$ to the
set of paths in the graph~$\partial_X\pv H$. It turns out that it
extends uniquely to a continuous function, which we denote $\chi_\pv
H$, from $\overline{[X]}$ to the set of pseudopaths of the same graph.
The following result provides a recursive criterion for the validity
of pseudoidentities in pseudovarieties of the form
$\pv{CR}\cap\bar{\pv H}$. The term ``recursive'' is used here in the
sense that the criterion for equality calls itself repeatedly on
pseudoidentities involving smaller contents.

\begin{Thm}[{\cite[Theorem~3.9]{Almeida&Trotter:2001}}]
  \label{t:wp-CRH}
  Let \pv H be a pseudovariety of groups and let $u,v\in\Om AS$. Then
  the pseudovariety $\pv{CR}\cap\bar{\pv H}$ satisfies the
  pseudoidentity $u=v$ if and only if each of the following conditions
  holds:
  \begin{enumerate}[(i)]
  \item\label{item:wp-CRH-1} $c(u)=c(v)$;
  \item\label{item:wp-CRH-2} the pseudoidentity $0(u)=0(v)$ holds
    in~$\pv{CR}\cap\bar{\pv H}$;
  \item\label{item:wp-CRH-3} the pseudoidentity $1(u)=1(v)$ holds
    in~$\pv{CR}\cap\bar{\pv H}$;
  \item\label{item:wp-CRH-4} either $|c(u)|=1$ and the pseudoidentity
    $u=v$ holds in~\pv H, or $|c(u)|>1$ and the pseudoidentity
    $\chi_\pv H(u)=\chi_\pv H(v)$ holds in~\pv H.
  \end{enumerate}
\end{Thm}

We are interested in distinguishing elements of~\Om AS that,
projected in $\Om A{}(\pv{CR}\cap\bar{\pv H})$, fall in the same
subgroup of the minimum ideal, where $A$ is a non-singleton finite
alphabet. For such elements, conditions
\eqref{item:wp-CRH-1}--\eqref{item:wp-CRH-3} of Theorem~\ref{t:wp-CRH}
are automatically fulfilled. Thus, the distinction must be done
through the condition of the pseudoidentity $\chi(u)=\chi(v)$ failing
in~\pv H. We want to do it under minimal assumptions, namely that
the pseudovariety \pv H~is nontrivial, say it contains $\pv{Ab}_p$,
where $p$~is prime. Indeed, it suffices to show that the (profinite)
numbers of occurrences in the two pseudowords in question of some
maximal factor of content missing just one letter, modulo equality
over~$\pv{CR}\cap\bar{\pv H}$, can be distinguished modulo~$p$.
Alternatively, we may work directly %
in~$\Om A{}(\pv{CR}\cap\bar{\pv H})$, which avoids the need to
consider the identification over $\pv{CR}\cap\bar{\pv H}$ of maximal
factors missing just one letter.

\begin{Thm}
  \label{t:ggm-CRH}
  Let \pv H be a nontrivial pseudovariety of groups and let $A$~be a
  nonempty finite alphabet. Then the semigroup %
  $\Om A{}(\pv{CR}\cap\bar{\pv H})$ is GGM whenever at least one of
  the following conditions holds:
  \begin{enumerate}[(i)]
  \item\label{item:ggm-CRH-1} $|A|\ne2$;
  \item\label{item:ggm-CRH-2} $\pv H\ne\pv{Ab}_2$.
  \end{enumerate}
\end{Thm}

\begin{proof}
  The case of a singleton alphabet $A$ is obvious, since then the
  semigroup $\Om A{}(\pv{CR}\cap\bar{\pv H})$ is a group. We therefore
  assume from hereon that $|A|\ge2$.
  
  In view of Theorem~\ref{t:weak-gm-CR},
  Proposition~\ref{p:reduction}\eqref{item:reduction-2}, and duality,
  it remains to show that, under the hypotheses \eqref{item:ggm-CRH-1}
  or~\eqref{item:ggm-CRH-2}, the minimum ideal $K$ of the
  semigroup~$\Om A{}(\pv{CR}\cap\bar{\pv H})$ has plenty of torsion on
  the left. So, suppose that $e,f\in K$ are distinct \Cl R-equivalent
  idempotents. Then we have $0(e)=0(f)$ and $\bar0(e)=\bar0(f)$ while
  $1(e)\ne1(f)$ or $\bar1(e)\ne\bar1(f)$. Note that
  $e=\bigl(0(e)\bar0(e)\bar1(e)1(e)\bigr)^\omega$ and
  $f=\bigl(0(e)\bar0(e)\bar1(f)1(f)\bigr)^\omega$, because each pair
  of idempotents in these equalities lie in the same \Cl H-class. We
  show that there is an idempotent $g$ from the \Cl L-class of~$e$
  such that $fg\ne eg$, which establishes that $K$ has plenty of
  torsion on the left.

  Let $x$~be an arbitrary element of~$\Om A{}(\pv{CR}\cap\bar{\pv
    H})$. Note that, if $xf$ is idempotent but $xe$ is not then, since
  they are \Cl R-equivalent by Green's Lemma, we obtain
  $xe(xe)^\omega=xe\ne(xe)^\omega=xf(xe)^\omega$ and so
  $g=(xe)^\omega$ has the desired property. Let $a=\bar1(e)$ and let
  $u$~be a word with $c(u)=A\setminus\{a\}$. Taking $x=(auf)^\omega$,
  for which $xf$ is idempotent, we conclude that we may assume that
  $xe$ is also idempotent. It follows that %
  \begin{equation}
    \label{eq:special-form-xe}
    xe=((auf)^\omega e)^\omega=(aua1(e))^\omega \  \text{and}\ 
    xf=(auf)^\omega=(au\bar1(f)1(f))^\omega .
  \end{equation}
  
  Suppose first that $\bar1(f)=a$. Then $1(e)$ and $1(f)$ must be
  distinct. We claim that it is possible to choose
  $v\in(A\setminus\{a\})^*$ such that, for
  $g_v=\bigl(va1(e)\bigr)^\omega$, we have $xeg_v\ne xfg_v$. Since
  $xfg_v$ and $xeg_v=xe$ lie in the same maximal subgroup of~$K$, to
  prove that they are distinct it suffices to show that there is some
  maximal factor $w$ that misses exactly the letter $a$ and such that
  the (profinite) numbers of times it appears in~$xfg_v$ and $xeg_v$
  are not congruent modulo~$p$. Because of the special 
  form~(\ref{eq:special-form-xe}) that the
  pseudowords $xe$  and $xf$ have, any two consecutive occurrences
  of~$a$ in $xeg_v$ and $xfg_v$ are
  separated by pseudowords of content $A\setminus\{a\}$. Thus, the
  factors in question are precisely those that appear between two
  consecutive occurrences of the letter~$a$, together with the factor
  after the last occurrence of~$a$. Table~\ref{tab:1} gives the number
  of occurrences of such factors, taken in the profinite completion of the
  additive group~$\mathbb{Z}$.
  \begin{table}[h]
    $$\begin{array}[t]{r|c|c}
      & %
      xeg_v=(aua1(e))^\omega(va1(e))^\omega %
      & %
      xfg_v=(aua1(f))^\omega(va1(e))^\omega %
      \\
      \hline
      u & 0 & \hphantom{-}0 \\
      1(e) & 0 & \hphantom{-}1 \\
      1(f) & 0 & -1 \\
      1(e)v & 0 & -1 \\
      1(f)v & 0 & \hphantom{-}1
    \end{array}
    $$
    \caption{}
    \label{tab:1}
  \end{table}

  Table~\ref{tab:1} does not take into account possible equalities
  between some of the elements in the first column, in which case the
  corresponding remainders of the rows should be summed. The possible
  equalities with $u$ may be ignored since the values in the
  corresponding row sum are the same. If we choose $v$ to be a
  letter then
  we guarantee the inequalities $1(e)v\ne1(e)$ and $1(f)v\ne1(f)$.
  Indeed,
  if $1(e)v=1(e)$ holds in $\pv{CR}\cap \pv{\bar H}$ then, in particular,
  $1(e)v=1(e)$ holds in $\pv H\subseteq \pv{CR}\cap \pv{\bar H}$.
  This gives, for a letter $v$, that $v=1$ holds in  $\pv H$.
  Hence $\pv H=\pv I$, a contradiction.
  
  In case $|A|=2$, so that $\pv H\ne\pv{Ab}_2$, let $v=b$ be the only
  letter in~$A\setminus\{a\}$. Then, since $1(e)\ne1(e)b,1(f)$, then
  the total for $1(e)$ in the column of Table~\ref{tab:1} headed by
  $xfg_v$ is either $1$ or~$2$, which is different from the null total
  corresponding to the other column.

  In case $|A|\ge3$, let $b=\mathrm{t}_1(1(e))$. If $b$~is also the
  last letter of~$1(f)$, and $c$~is a letter from~$A\setminus\{a,b\}$,
  then the three pseudowords $1(e),1(f),1(e)c$ are distinct, where the
  inequality $1(f)\ne1(e)c$ follows from the fact that the two sides
  end with different letters. Taking $v=c$, of the four elements
  $1(e),1(f),1(e)v,1(f)v$, at least one is not equal to any of the
  others and the corresponding row in Table~\ref{tab:1} shows that
  $xeg_v\ne xfg_v$. If $c=\mathrm{t}_1(1(f))\ne b$, then similarly the
  three pseudowords $1(e),1(f),1(f)c$ are distinct, and the same
  argument applies.
  
  It remains to consider the case where the letter $b=\bar1(f)$ is
  such that $b\ne a$, so that $a$~occurs in~$1(f)$. Note that there is
  a factorization $1(f)=(1(f))^{\omega+1}=w_0aw_1aw_2$, where $c(w_0)$
  and $c(w_2)$ are both contained in~$A\setminus\{a,\bar1(f)\}$.
  Proceeding as in the preceding case, we obtain Table~\ref{tab:2},
  provided we take $v$ such that $c(v)=A\setminus\{a\}$, where we take
  into account that the contribution of the factors in question that
  appear within $w_1$ is null because they appear $\omega$ times,
  while there are none within $w_2$ because $a,b\notin c(w_2)$.
  \begin{table}[h]
    $$\begin{array}[t]{r|c|c}
      & %
      xeg_v=(aua1(e))^\omega(va1(e))^\omega %
      & %
      xfg_v=(aubw_0aw_1aw_2)^\omega(va1(e))^\omega %
      \\
      \hline
      u & 0 & \hphantom{-}0 \\
      1(e) & 0 & \hphantom{-}1 \\
      1(e)v & 0 & -1 \\
      ubw_0 & 0 & \hphantom{-}0 \\
      w_2v & 0 & \hphantom{-}1
    \end{array}
    $$
    \caption{}
    \label{tab:2}
  \end{table}
  
  Since the numeric values of the sum of all the rows are distinct in
  every nontrivial cyclic group, we deduce that $xeg_v\ne xfg_v$,
  which completes the proof of the theorem.
\end{proof}

Using Theorem~\ref{t:wp-CRH}, one may check that the two elements
$(ab)^\omega$ and $(ab^2)^\omega$ of the minimum ideal %
of~$\Om{\{a,b\}}(\pv{CR}\cap\overline{\pv{Ab}_2})$ have the same image
under~$\lambda^K$.

The following result is less precise than Theorem~\ref{t:ggm-CRH} but
sufficient for the applications in Sections~\ref{sec:orderability}
and~\ref{sec:irreducibility}.

\begin{Cor}
  \label{c:ggm-CRH}
  For a nontrivial pseudovariety of groups \pv H, the pseudovariety
  $\pv{CR}\cap\bar{\pv H}$~is GGM, unless $\pv H=\pv{Ab}_2$, in which
  case it is almost GGM.\qed
\end{Cor}

\section{Orderability and order primitivity}
\label{sec:orderability}

A partial order on a set $S$ is said to be \emph{trivial} if it is the
equality relation on~$S$. By a \emph{(partially) ordered semigroup} we
mean a semigroup endowed with a stable partial ordering. A
\emph{pseudovariety of ordered semigroups} is a nonempty class of
finite ordered semigroups which is closed under taking images under
order-preserving homomorphisms, subsemigroups with the induced
ordering, and finite direct products (under the component-wise
ordering). When we talk about the pseudovariety of semigroups
generated by a class \Cl C of finite ordered semigroups, we mean the
pseudovariety of semigroups generated by the members of~\Cl C, for
which the order is ignored. On the other hand, every semigroup can be
viewed as an ordered semigroup for the trivial ordering, reduced to
equality. For a pseudovariety \pv V of semigroups, the pseudovariety
of ordered semigroups $\pv V_\mathrm{o}$ it generates consists
precisely of the members of~\pv V under all possible stable partial
orders. It is common practice in the literature to identify \pv V with
$\pv V_\mathrm{o}$.

From the point of view of the applications in computer science, the
interest in pseudovarieties of ordered semigroups stems from the fact
that the corresponding \emph{positive varieties} of regular languages
are defined similarly to Eilenberg's varieties of languages by
dropping the requirement of closure under
complementation from the definition of varieties of languages. This
has prompted the investigation of many pseudovarieties of ordered
semigroups and it is natural to ask what pseudovarieties of semigroups
they generate. The origin of the work reported in this paper lays
precisely at an
attempt to answer this kind of question. Other than the application
of some of the representation results from previous sections, this
section contains only elementary observations. Its main purpose is to
show that several familiar pseudovarieties cannot be obtained in that
way.

 We say that a topological semigroup is
\emph{orderable} if it admits a nontrivial closed stable partial
order. The following is a simple extension to the profinite case of
a well-known property of finite groups.

\begin{Lemma}
  \label{l:profinite-groups-unorderable}
  No profinite group is orderable.
\end{Lemma}

\begin{proof}
  Let $G$ be a profinite group and let $\le$ be a closed stable
  partial order on~$G$. Let $g\in G$ be such that $1\le g$. By
  stability of the partial order, the relation $g^n\le g^{n+1}$ holds
  for every positive integer~$n$. Hence, the inequality $g\le g^n$
  holds for every positive integer~$n$. Considering in particular the
  relations $g\le g^{n!}$, we deduce that, since the order is closed
  and $\lim g^{n!}=1$, the relation $g\le 1$ also holds. Since $\le$
  is assumed to be a partial order, it follows that $g=1$. It follows
  that the relation $\le$ is trivial.
\end{proof}

In contrast with Lemma~\ref{l:profinite-groups-unorderable}, we have
the following simple observation.

\begin{Lemma}
  \label{l:Om1V-orderable}
  Let \pv V be a pseudovariety which is not contained in~\pv G. Then
  $\Om{\{a\}}V$~is orderable.
\end{Lemma}

\begin{proof}
  Consider the relation $\le$ defined by $u\le v$ if either $u=v$, or
  $u=a^n$ and $v=a^{\omega+n}$, where $n$~is a positive integer. One
  can easily check that it is a closed stable partial order
  on~$\Om{\{a\}}V$. By hypothesis, it is nontrivial.
\end{proof}

For the sequel, we need the following simple auxiliary lemma.

\begin{Lemma}[{\cite[Lemma~4.6.23]{Rhodes&Steinberg:2009qt}}]
  \label{l:GGM-max-subgroups}
  Let $S$ be a nontrivial GGM profinite semigroup, with minimum ideal
  $K$. Then, given distinct elements $s,t\in S$, there exist $x,y\in
  K$ such that $xsy\ne xty$. In particular, the maximal subgroups
  of~$K$ are nontrivial.
\end{Lemma}

We say that the pseudovariety \pv V is \emph{almost unorderable} if
the semigroups \Om AV are unorderable for finite alphabets $A$ with
$|A|$ arbitrarily large. If \Om AV is unorderable for every finite
set~$A$ with at least two elements, then we say that \pv V~is
\emph{unorderable}. We say \pv V is \emph{strictly orderable} if, for
every finite nonempty set $A$, \Om AV admits a nontrivial closed
stable partial order. Thus, a strictly orderable pseudovariety is not
almost unorderable. We do not know if the converse is true.

The next proposition relates unorderability with the GGM property.

\begin{Prop}
  \label{p:GGM-unorderable}
  Let $S$ be a nontrivial GGM profinite semigroup. Then $S$~is
  unorderable.
\end{Prop}

\begin{proof}
  Suppose that $\le$ is a closed stable partial order on~$S$ for which
  there are elements $s,t\in S$ such that $s<t$. By
  Lemma~\ref{l:GGM-max-subgroups}, the minimum ideal of~$S$ contains
  elements $x$ and $y$ such that $xsy\ne xty$. Since the relation
  $\le$~is stable, it follows that $xsy<xty$. Hence the induced order
  on the maximal subgroup $H$ of~$S$ containing both $xsy$ and $xty$
  is a nontrivial closed stable partial order on the profinite
  group~$H$, which contradicts
  Lemma~\ref{l:profinite-groups-unorderable}. Thus, $S$~is
  unorderable.
\end{proof}

The following is an immediate corollary of
Proposition~\ref{p:GGM-unorderable}.

\begin{Cor}
  \label{c:GGM-unorderable}
  \begin{enumerate}[(a)]
  \item\label{item:GGM-unorderable-1} Every GGM pseudovariety is
    unorderable.
  \item\label{item:GGM-unorderable-2} Every almost GGM pseudovariety is
    almost unorderable.\qed
  \end{enumerate}
\end{Cor}

Combining Corollary~\ref{c:GGM-unorderable} with results from other
sections,
we
obtain many familiar examples of unorderable pseudovarieties.

\begin{Cor}
  \label{c:gm-unorderable-egs}
  Let \pv H be a nontrivial pseudovariety of groups. Then the
  pseudovarieties \pv H, $\bar{\pv H}$,
  $\pv{CS}\cap\bar{\pv H}$ are unorderable. So is $\pv{CR}\cap\bar{\pv
    H}$, except in the case of $\pv H=\pv{Ab}_2$, for which it is
  almost unorderable.
\end{Cor}

\begin{proof}
  In each case, it suffices to justify that the pseudovariety in
  question is GGM, or almost GGM in the exceptional case. For \pv H,
  this is obvious, but the unorderability also follows directly from
  Lemma~\ref{l:profinite-groups-unorderable}. For $\bar{\pv H}$, the
  GGM property is given by Corollary~\ref{c:gm-barH}. For
  $\pv{CR}\cap\bar{\pv H}$, it suffices to invoke
  Corollary~\ref{c:ggm-CRH}. For $\pv{CS}\cap\bar{\pv H}$, the GGM
  property follows from the structure theorem for free profinite
  semigroups over this pseudovariety, which entails that they are full
  of torsion. The case of $\pv H=\pv G$ has been studied in detail
  in~\cite{Almeida:1991d} but the arguments and results apply equally
  well by replacing \pv G by a nontrivial pseudovariety of groups \pv
  H.
\end{proof}

It should be noted that there are also important pseudovarieties which
are strictly orderable. Such an example is given by the pseudovariety
\pv J. The following is an easy consequence of the results
of~\cite[Section~8.2]{Almeida:1994a}.

\begin{Prop}
  \label{p:J-orderable}
  The pseudovariety \pv J is strictly orderable.
\end{Prop}

\begin{proof}
  Let $A$ be a finite nonempty set. For $u,v\in\Om AJ$, let $u\le v$
  if every (finite) subword of~$u$ is also a subword of~$v$. It is
  routine to check that $\le$ is a closed stable quasi-order on~\Om
  AJ. By \cite[Theorem~8.2.8]{Almeida:1994a}, it is a partial order.
\end{proof}

There is a connection between orderability and pseudovarieties of
ordered semigroups that we proceed to analyze.

We say that a pseudovariety \pv V of semigroups is \emph{order
  primitive} if there is no pseudovariety of ordered semigroups
properly contained in~\pv V that generates \pv V as a pseudovariety of
semigroups.

One of the formulations of Simon's characterization of piecewise
testable languages \cite{Simon:1975} is the theorem of Straubing and
Th\'erien \cite{Straubing&Therien:1988a} stating that the pseudovariety
of all finite \Cl J-trivial monoids is generated by the pseudovariety
consisting of all finite ordered monoids satisfying the inequality
$x\le1$. It is easy to deduce that \pv J is generated by the
pseudovariety of all finite ordered semigroups satisfying the
inequalities $xy\le y$ and $yx\le y$. Hence, \pv J is not order
primitive.

\begin{Thm}
  \label{t:almost-unorderable-implies-order-primitive}
  Every almost unorderable pseudovariety of semigroups is order
  primitive.
\end{Thm}

\begin{proof}
  Let \pv V be an almost unorderable pseudovariety of semigroups and
  let \pv U be a pseudovariety of ordered semigroups properly
  contained in~\pv V. By the version of Reiterman's Theorem for
  pseudovarieties of ordered semigroups
  \cite{Molchanov:1994,Pin&Weil:1996b}, there is a finite alphabet $A$
  and there are distinct $u,v\in\Om AV$ such that \pv U satisfies the
  inequality $u\le v$. Since \pv V~is almost unorderable, we may
  assume that \Om AV~is unorderable.

  Consider the relation $\preceq$ on~\Om AV defined by $w\preceq z$ if
  \pv U~satisfies the inequality $w\le z$. Note that it is a closed
  stable quasi-order on~\Om AV. Since \Om AV~is unorderable, it
  follows that $\preceq$ fails the only missing property to be a
  closed stable partial order, namely anti-symmetry. Hence, there are
  distinct elements $w,z\in\Om AV$ such that $w\preceq z$ and
  $z\preceq w$, that is \pv U~satisfies the pseudoidentity $w=z$,
  which fails in~\pv V. Thus, \pv U generates a proper
  subpseudovariety of semigroups of~\pv V.
\end{proof}

Note that the two-element left-zero semigroup, with the trivial order,
generates \pv{LZ} as a pseudovariety of ordered semigroups. A
pseudovariety \pv V of ordered semigroups that generates \pv{LZ}, as a
pseudovariety of semigroups, must contain a two-element left-zero
semigroup, with some stable partial order. Since the product of two
copies of this semigroup contains the two element left-zero semigroup
with the trivial order, we deduce that $\pv V=\pv{LZ}$, and so
\pv{LZ}~is order primitive. Note that every partial order on a
left-zero semigroup is stable. Hence \pv{LZ}~is strictly orderable,
which shows that the converse of
Theorem~\ref{t:almost-unorderable-implies-order-primitive} is false.

Combining Theorem~\ref{t:almost-unorderable-implies-order-primitive}
with Corollary~\ref{c:gm-unorderable-egs}, we obtain the following
result.

\begin{Cor}
  \label{c:order-primitive-egs}
  Let \pv H be a nontrivial pseudovariety of groups. Then the
  pseudovarieties of the form \pv H, $\bar{\pv H}$,
  $\pv{CR}\cap\bar{\pv H}$, and $\pv{CS}\cap\bar{\pv H}$, are order
  primitive.\qed
\end{Cor}

A stronger result for the pseudovarieties $\bar{\pv H}$,
$\pv{CR}\cap\bar{\pv H}$ is obtained in
Section~\ref{sec:irreducibility}, which includes the pseudovariety \pv
A. The structure of the lattice of varieties of ordered bands has been
completely determined by Ku\v ril \cite{Kuril:2015}. The main
ingredient is to show that every variety of ordered bands that is not
a variety of bands is actually a variety of ordered normal bands,
which have been completely identified by Emery \cite{Emery:1999}. One
may easily deduce that the pseudovariety $\pv B=\pv{CR}\cap\pv A$ is
order primitive. For $\pv{RB}=\pv{CS}\cap\pv A$, one can easily deduce
that it is order primitive from the discussion above concerning
\pv{LZ} and the dual result for~\pv{RZ}.

\section{Join irreducibility}
\label{sec:irreducibility}

Following \cite[Definition~6.1.5]{Rhodes&Steinberg:2009qt}, we say
that an element $s$ of a lattice is \emph{strictly finite join
  irreducible (sfji)} if, whenever $s=t\vee u$, either $s=t$ or $s=u$;
the element $s$~is \emph{finite join irreducible (fji)} if, whenever
$s\le t\vee u$, either $s\le t$ or $s\le u$. Note that fji implies
sfji.

The element $s$ of a lattice is \emph{meet-distributive} if the
equality $s\wedge(t\vee u)=(s\wedge t)\vee(s\wedge u)$ holds for all
$t$ and $u$ in the lattice. Note that every sfji meet-distributive
element of a lattice is fji.

The lattices of concern in this paper, which are both complete, are
$\Cl L(\pv S)$, of all pseudovarieties of semigroups, and $\Cl L_o(\pv
S)$, of all pseudovarieties of ordered semigroups, both lattices
ordered by inclusion. One can easily check that $\Cl L(\pv S)$ is a
complete sublattice of~$\Cl L_o(\pv S)$. The above lattice notions are
always to be understood here with respect to one of these lattices. Of
course, for an element of~$\Cl L(\pv S)$, being sfji or fji with
respect to~$\Cl L_o(\pv S)$ are stronger properties than their
counterparts within the lattice $\Cl L(\pv S)$. An example of an sfji
pseudovariety which is not fji in~$\Cl L(\pv S)$ can be found
in~\cite[Proposition~7.3.22]{Rhodes&Steinberg:2009qt}.

The following two theorems open a second path to applications of the
representation results of the preceding sections.

\begin{Thm}
  \label{t:fji}
  Let \pv V be an almost WGGM pseudovariety of semigroups such that at
  least one of the following conditions holds:
  \begin{enumerate}[(i)]
  \item\label{item:fji-1} $\pv V\bullet\pv{Sl}=\pv V$;
  \item\label{item:fji-2} \pv V contains~\pv{Sl} and $\pv
    V\bullet\pv{Ab}_p=\pv V$ for some prime $p$.
  \end{enumerate}
  Then \pv V is fji in the lattice $\Cl L_o(\pv S)$.
\end{Thm}

\begin{proof}
  Let \pv U and \pv W be pseudovarieties of ordered semigroups and
  suppose that \pv V is contained in $\pv U\vee\pv W$ but in neither
  \pv U nor~\pv W. Since pseudovarieties of ordered semigroups are
  defined by inequalities, there is a finite alphabet $A$ and there
  are pseudowords $u_1,u_2,w_1,w_2\in\Om AS$ such that the inequality
  $u_1\le u_2$ holds in~\pv U, $w_1\le w_2$ holds in~\pv W, and both
  inequalities (and therefore also the pseudoidentities $u_1=u_2$ and
  $w_1=w_2$) fail in~\pv V.

  Without loss of generality, we may assume that the sets $c(u_1)\cup
  c(u_2)$ and $c(w_1)\cup c(w_2)$ are disjoint and do not contain the
  letter~$z\in A$ and, furthermore, that \Om AV~is WGGM: otherwise, we
  rewrite one of the pseudoidentities in a new, disjoint, alphabet,
  and add it to~$A$ together with enough extra letters, including~$z$,
  using the hypothesis that \pv V is almost WGGM. Let
  $B=A\setminus\{z\}$ and let $\pi:\Om AS\to\Om AV$ be the natural
  continuous homomorphism, mapping free generators to themselves.

  Since the pseudoidentity $u_1=u_2$ fails in~\pv V, there exists a
  continuous homomorphism $\varphi:\Om BS\to S$ into a semigroup
  $S\in\pv V$ such that $\varphi(u_1)\ne\varphi(u_2)$.
  Let $U_1$~be
  the two-element semilattice, $\xi:S^1\to U_1$ be the mapping that
  sends $\varphi(u_1)$ to~$0$ and every other element to~$1$, and
  $\psi$~be the extension of~$\varphi$ to a continuous homomorphism
  $\Om AS\to M(S,U_1,\xi)$ that maps $z$ to~$(1,1,1)$. Then
  $\psi(zu_1z)=(1,0,1)\ne(1,1,1)=\psi(zu_2z)$ and so the
  pseudoidentity $zu_1z=zu_2z$ fails in~$\pv V\bullet\pv{Sl}$.
  Similarly, simply replacing $U_1$ by the additive group
  $\mathbb{Z}/p\mathbb{Z}$, the pseudoidentity $zu_1z=zu_2z$ fails
  in~$\pv V\bullet\pv{Ab}_p$.
  Thus, under the hypotheses of the theorem, the pseudoidentity
  $zu_1z=zu_2z$ fails in~\pv V, and the same argument and conclusion
  applies to the pseudoidentity $zw_1z=zw_2z$.

  Because they all have proper content, none of the pseudowords
  $\pi(zu_1z)$, $\pi(zu_2z)$, $\pi(zw_1z)$, $\pi(zw_2z)$ belongs to
  the minimum ideal $K$ of~\Om AV. Since \Om AV~is WGGM and $\pi$ maps
  the minimum ideal $I$ of~\Om AS onto~$K$
  \cite[Lemma~4.6.10]{Rhodes&Steinberg:2009qt}, there exist $s,t\in I$
  such that $\pi(szu_1z)\ne\pi(szu_2z)$ and
  $\pi(zw_1zt)\ne\pi(zw_2zt)$. Let $u_i'=szu_iz$ and $w_i'=zw_izt$ for
  $(i=1,2)$. Note that the relations $u_1'\mathrel{\Cl R}u_2'$ %
  and $w_1'\mathrel{\Cl L}w_2'$ hold. Moreover, by Green's Lemma, the
  inequalities $w_2'u_1'\le w_2'u_2'$ and $w_1'u_2'\le w_2'u_2'$ are
  also nontrivial in~\pv V, while they remain valid respectively
  in~\pv U and~\pv W. Let $v=w_2'u_2'$. Furthermore, multiplying
  $w_2'u_1'\le v$ on the left by~$v^{\omega-1}$ and $w_1'u_2'\le v$
  on the right by $v^{\omega-1}$ we obtain the pseudowords
  \begin{align}
    \label{eq:fji-u}
    u & %
    = v^{\omega-1}w_2'u_1' %
    = (zw_2ztszu_2z)^{\omega-1}zw_2ztszu_1z
    \\
    \label{eq:fji-v}
    e & %
    = v^\omega %
     =(zw_2ztszu_2z)^\omega
    \\
    \label{eq:fji-w}
    w & %
    = w_1'u_2'v^{\omega-1} %
    = zw_1ztszu_2z(zw_2ztszu_2z)^{\omega-1}
  \end{align}
  such that the following conditions hold:
  \begin{align}
    \label{eq:fji-1}
    & %
    u\mathrel{\Cl R}e\mathrel{\Cl L}w,\
    e^2=e,\\
    \label{eq:fji-2}
    & %
    \pv U\models u\le e,\ %
    \pv W\models w\le e,\\
    \label{eq:fji-3}
    & %
    \pi(u)\ne\pi(e)\ne\pi(w).
  \end{align}
  
  We claim that, under the assumption that the
  condition~\eqref{item:fji-2} holds, so does the following:
  \begin{equation}
    \label{eq:fji-3b}
    \pi(wu)\ne\pi(e).
  \end{equation}
  To prove the claim, consider a prime~$p$ such that $\pv
  V\bullet\pv{Ab}_p=\pv V$. By the choice of the pseudowords
  $u_1,u_2,w_1,w_2$, there exists a continuous homomorphism
  $\varphi:\Om BS\to S$ into a semigroup $S\in\pv V$ such that
  \begin{equation}
    \label{eq:fji-3c}
    \varphi(w_1)\notin\{\varphi(u_1),\varphi(u_2),\varphi(w_2)\}.
  \end{equation}
  Note that, since \pv V contains the semilattice $U_1$, whence the
  semigroup $S\times U_1$, we may assume that $\varphi(w_1)\ne1$. Let
  $\xi:S^1\to\mathbb{Z}/p\mathbb{Z}$ map $\varphi(w_1)$ to~$1$ and
  every other element to~$0$ and let $\psi:\Om AS\to
  M(S,\mathbb{Z}/p\mathbb{Z},\xi)$ be the extension of~$\varphi$ to a
  continuous homomorphism that maps $z$ to~$(1,0,1)$. Since $\psi(z)$
  is an idempotent and $u,w,e$ start and end with~$z$, the values of
  $wu$ and~$e$ under~$\psi$ are both of the form $(1,g,1)$. Since
  $\xi\bigl(\varphi(w_1)\bigr)=1$, $\psi(zw_1z)$ is~$(1,1,1)$ while,
  by~\eqref{eq:fji-3c}, we have $\psi(zu_iz)=\psi(zw_2z)=(1,0,1)$
  ($i=1,2$). Let $\psi(ztsz)=(1,h,1)$. Then, taking into account the
  expressions~\eqref{eq:fji-u}--\eqref{eq:fji-w} and the fact that
  $\psi$ is a continuous homomorphism, we may compute
  $$\psi(wu) %
  = (1,1,1)(1,h,1)^\omega %
  = (1,1,1) %
  \ne (1,0,1) %
  = (1,h,1)^\omega %
  = \psi(e). %
  $$
  This establishes the claim since $M(S,\mathbb{Z}/p\mathbb{Z},\xi)$
  belongs to~$\pv V\bullet\pv{Ab}_p$ and, therefore, to~\pv V.

  Consider next the following inequality, where $y$~is a new letter:
  \begin{align}
    \label{eq:fji-a}
    y(uy)^{\omega-1} wuy (wy)^{\omega-1} %
    \le %
    y(uy)^\omega (ey)^{\omega-1} (wy)^\omega.
  \end{align}
  Let $C=A\cup\{y\}$. We claim that \eqref{eq:fji-a} holds in both \pv
  U and \pv W. Since the arguments for the two pseudovarieties are
  dual, we treat only the case of~\pv U. In view of~\eqref{eq:fji-2},
  \pv U~satisfies the inequality $wu\le we$ and thus also $wu\le w$
  because $we=w$ by~\eqref{eq:fji-1}. Hence, \pv
  U~satisfies the following inequalities:
  \begin{equation*}
    y(uy)^{\omega-1} wuy (wy)^{\omega-1}
    \le %
    y(uy)^\omega(uy)^{\omega-1} wy (wy)^{\omega-1}
    \mathrel{\mathop\le\limits_{{}_{(\ref{eq:fji-2})}}} %
    y(uy)^\omega(ey)^{\omega-1} (wy)^\omega,
  \end{equation*}
  which establishes that \eqref{eq:fji-a}~holds in~\pv U. We will
  reach a contradiction by showing that \eqref{eq:fji-a} does not hold
  in~\pv V, which is contrary to the assumption that $\pv V\subseteq\pv
  U\vee\pv W$.
  
  Suppose first that condition~\eqref{eq:fji-3b} holds, which we have
  not proved under the hypothesis~\eqref{item:fji-1}. Taking also into
  account~\eqref{eq:fji-3}, it follows that there is a continuous
  homomorphism $\varphi:\Om AS\to S$ into a semigroup $S\in\pv V$ such
  that $\varphi(e)\notin\{\varphi(u),\varphi(w),\varphi(wu)\}$.
  Again, we may assume that $\varphi(e)\ne 1$.

  Let the function $\xi:S^1\to U_1$ map $\varphi(e)$ to~$0$ and every
  other element to~$1$. Consider the extension of~$\varphi$ to a
  continuous homomorphism $\psi:\Om CS\to M(S,U_1,\xi)$ that sends $y$
  to~$(1,1,1)$. Then we may compute
  \begin{align*}
    \psi\bigl(y(uy)^{\omega-1} wuy (wy)^{\omega-1}\bigr) %
    &= %
    (1,\xi(\varphi(u))\xi(\varphi(wu))\xi(\varphi(w)),1) %
    = (1,1,1), %
    \\
    \psi\bigl(y(uy)^\omega (ey)^{\omega-1} (wy)^\omega\bigr) %
    &= %
    (1,\xi(\varphi(u))\xi(\varphi(e))\xi(\varphi(w)),1) %
    = (1,0,1).
  \end{align*}
  Since $M(S,U_1,\xi)\in\pv V\bullet\pv{Sl}$, under the hypothesis
  that the condition \eqref{item:fji-1}~holds we deduce that the
  inequality \eqref{eq:fji-a} is not valid in~\pv V.

  On the other hand, if $p$~is a prime such that %
  $\pv V\bullet\pv{Ab}_p=\pv V$, then we consider the additive group
  $\mathbb{Z}/p\mathbb{Z}$ and the mapping
  $\delta:S^1\to\mathbb{Z}/p\mathbb{Z}$ that sends $\varphi(e)$ to~$1$
  and every other element to~$0$. Now, for the extension of~$\varphi$
  to a continuous homomorphism $\chi:\Om CS\to
  M(S,\mathbb{Z}/p\mathbb{Z},\delta)$ that sends $y$ to~$(1,0,1)$, we
  may compute
  \begin{align*}
    \nonumber
    \chi\bigl(y(uy)^{\omega-1} wuy (wy)^{\omega-1}\bigr) %
    &= %
    (1,-\delta(u)+\delta(wu)-\delta(w),1) %
    = (1,0,1), %
    \\
    \label{eq:fji-4}
    \chi\bigl(y(uy)^\omega (ey)^{\omega-1} (wy)^\omega\bigr) %
    &= %
    (1,-\delta(e),1) %
    = (1,-1,1).
  \end{align*}
  Since %
  $M(S,\mathbb{Z}/p\mathbb{Z},\delta)\in\pv V\bullet\pv{Ab}_p=\pv V$,
  it follows that the inequality \eqref{eq:fji-a}~is not valid in~\pv
  V, which again contradicts the assumption that $\pv V\subseteq\pv
  U\vee\pv W$.

  It remains to treat the case where $\pi(wu)=\pi(e)$ under the
  hypothesis~\eqref{item:fji-1}. Then the set
  $\pi\bigl(\{e,u,w,wu\}\bigr)$ is contained in a maximal subgroup
  of~$K$, which must therefore be nontrivial.\footnote{The reader
    interested only in the applications in
    Corollary~\ref{c:gm-unorderable-egs} and its sequel may skip the
    remainder of the proof, since the single application for which
    only the hypothesis \eqref{item:fji-1} can be used that we have in
    mind is the pseudovariety \pv A.
    } In this case,
  we may further replace $u$ by $ue$ and $w$ by~$ew$ without affecting
  any of the conditions \eqref{eq:fji-1}--\eqref{eq:fji-3} and so we
  may assume that the pseudowords $e,u,w$ lie in the same subgroup $H$
  of~$I$. Since $\pi(wu)=\pi(e)$, then $\pi(w)=\pi(u^{\omega-1})$ and
  we could replace $u$ by $u^{\omega-1}$ without affecting the
  conditions \eqref{eq:fji-1}--\eqref{eq:fji-3}. Thus, we may assume
  that $\pi(u)=\pi(w)$.

  For a pseudovariety \pv X, consider the relation on~\Om AS defined by
  $$p\preceq_\pv Xq %
  \quad\mbox{if } %
  \pv X\models p\le q.$$
  Note that $\preceq_\pv X$ is a stable quasi-order. The induced
  relation $\preceq$ on the profinite group~$H$ is in fact a
  nontrivial closed stable quasi-order. The binary relation given by
  ${\equiv}={\preceq}\cap{\succeq}$ is therefore a closed congruence
  on the profinite group $H$ and $\preceq$ induces a closed stable
  partial order on the quotient $H/{\equiv}$, which is a profinite
  group under the quotient topology. By
  Lemma~\ref{l:profinite-groups-unorderable}, it follows that the
  induced order on the group $H/{\equiv}$ is trivial, that is the
  relation $\preceq$ on~$H$ is the congruence~$\equiv$. Thus, we must
  have $\pv U\models u=e$ and $\pv W\models w=e$.

  From the conclusion of the preceding paragraph, we deduce that $\pv
  U\vee\pv W$, and therefore also \pv V, satisfies the pseudoidentity
  \begin{equation}
    \label{eq:fji-b}
    (uy)^\omega(wy)^\omega=(uy)^\omega(ey)^\omega(wy)^\omega.
  \end{equation}
  However, a calculation in the semigroup $M(S,U_1,\xi)$ considered
  above shows that $\psi\bigl((uy)^\omega(vy)^\omega\bigr)=(\varphi(u),1,1)$,
  while $\psi\bigl((uy)^\omega(ey)^\omega(vy)^\omega\bigr)=(\varphi(u),0,1)$,
  so that the pseudoidentity \eqref{eq:fji-b}~fails in~$\pv
  V\bullet\pv{Sl}$. Hence, under the hypothesis that \pv V~satisfies
  condition~\eqref{item:fji-1}, we obtain a contradiction, which
  completes the proof of Theorem~\ref{t:fji}.
\end{proof}

As examples of application of Theorem~\ref{t:fji}, taking into account
the previous WGGM results (namely Corollaries~\ref{c:barH-wggm},
\ref{c:wggm-DS-egs}, and Theorem~\ref{t:weak-gm-CR}) and
Proposition~\ref{p:bullet-group}, we obtain the join irreducibility of
many familiar pseudovarieties.

\begin{Cor}
  \label{c:fji-egs}
  The pseudovarieties \pv A, $\bar{\pv H}$, $\pv{DS}\cap\bar{\pv H}$,
  and $\pv{CR}\cap\bar{\pv H}$ are fji in the lattice $\Cl L_o(\pv S)$
  for every nontrivial pseudovariety of groups \pv H.\qed
\end{Cor}

That the pseudovarieties of the form $\bar{\pv H}$ are sfji in $\Cl
L(\pv S)$ for \pv H an extension closed pseudovariety of groups was
first proved in~\cite{Margolis&Sapir&Weil:1995}; this is in fact
deduced from the stronger property that such a pseudovariety $\bar{\pv
  H}$ cannot be expressed as a Mal'cev product of proper
subpseudovarieties, which also entails the similar property for
semidirect product.
In~\cite[Corollary~7.4.23]{Rhodes&Steinberg:2009qt}, it was proved the
finite join irreducibility in $\Cl L(\pv S)$ of the pseudovarieties of
the forms $\bar{\pv H}$, $\pv{DS}\cap\bar{\pv H}$, and
$\pv{CR}\cap\bar{\pv H}$, where \pv H~is a pseudovariety of groups
containing some non-nilpotent group. In~\cite{Almeida&Klima:2011a}, we
improved the results from~\cite{Margolis&Sapir&Weil:1995} by showing
that, for an arbitrary pseudovariety of groups \pv H, if a pseudovariety
of the form $\bar{\pv H}$ can be covered by a Mal'cev product of
pseudovarieties then at least one of them must contain~$\bar{\pv H}$,
which again entails
the similar property for semidirect product and join.

Another sufficient condition for join irreducibility is provided by
the following result.

\begin{Thm}
  \label{t:fji-equidivisible}
  Let \pv V be a pseudovariety closed under concatenation that contains
  some nontrivial group. Then \pv V is fji in the lattice $\Cl L_o(\pv
  S)$.
\end{Thm}

\begin{proof}
  We first note that \pv V~is WGGM by Corollary~\ref{c:wggm-AmV}. The
  proof now follows basically the same steps as the above proof of
  Theorem~\ref{t:fji}, with appropriate modifications when the closure
  properties assumed in the hypothesis of Theorem~\ref{t:fji} are
  invoked. We therefore adopt the same notation without further
  comment.

  The first modification concerns the justification of the fact that
  the pseudoidentities $zu_1z=zu_2z$ and $zw_1z=zw_2z$ fail in~\pv V,
  which is immediate from the hypothesis that $u_1=u_2$ and $w_1=w_2$
  fail in~\pv V, taking into account that \pv V~is equidivisible.

  The second modification which is needed is to justify the
  inequality~\eqref{eq:fji-3b}. This
  is done again by invoking equidivisibility of~\pv V and observing
  that $wu$ admits $zw_1z$ as a prefix, whereas $zw_2z$ is a prefix
  of~$e$, where $z$~is a letter that does not occur in $w_1$ and $w_2$.

  Finally, it remains to show that the hypothesis that the inequality
  \eqref{eq:fji-a} holds in~\pv V leads to a contradiction. Since \pv
  V~is a pseudovariety of semigroups, that hypothesis means that we
  have the following two factorizations of the same element of~\Om CV,
  where we write $\bar r$ for $\pi(r)$ with $r\in\Om AS$:
  \begin{equation}
    \label{eq:fji-equidivisible-1}
    y(\bar uy)^{\omega-1}\cdot \bar w\bar uy (\bar wy)^{\omega-1} %
    = %
    y(\bar uy)^\omega \cdot(\bar ey)^{\omega-1} (\bar wy)^\omega.
  \end{equation}
  By equidivisibility of~\Om CV, there is some $q\in(\Om CV)^1$ such
  that one of the following conditions holds:
  \begin{align}
    \label{eq:fji-equidivisible-2}
    &y(\bar uy)^{\omega-1}q=y(\bar uy)^\omega %
    \quad\mbox{and}\quad %
    \bar w\bar uy(\bar wy)^{\omega-1} %
    = %
    q(\bar ey)^{\omega-1} (\bar wy)^\omega,
    \\
    \label{eq:fji-equidivisible-3}
    &y(\bar uy)^{\omega-1}=y(\bar uy)^\omega q%
    \quad\mbox{and}\quad %
    q\bar w\bar uy(\bar wy)^{\omega-1} %
    = %
    (\bar ey)^{\omega-1} (\bar wy)^\omega.
  \end{align}

  By hypothesis, \pv V~contains some additive group of the form
  $\mathbb{Z}/p\mathbb{Z}$, where $p$~is a prime. Let
  $m\in\{0,1,\ldots,p-1\}$ be such that $\varphi(q)=m$,
  where $\varphi:\Om CV\to\mathbb{Z}/p\mathbb{Z}$ is the unique
  continuous homomorphism that maps $y$ to~$1$ and every other
  element of~$C$ to~$0$. From the first equalities
  in~\eqref{eq:fji-equidivisible-2}
  and~\eqref{eq:fji-equidivisible-3}, we deduce that we must have,
  respectively, $m=1$ and $m=p-1$. In particular, $y$ occurs at least
  once in~$q$. Consider the unique factorization of the form
  $q=q_0yq_1$, where $y\notin c(q_0)$ and $q_0,q_1\in(\Om CV)^1$,
  where existence follows from compactness, and uniqueness from
  equidivisibility. From the second equalities
  in~\eqref{eq:fji-equidivisible-2} and~\eqref{eq:fji-equidivisible-3}
  and equidivisibility of~\Om CV, we deduce, respectively, that
  $q_0=\bar w\bar u$ and $q_0=\bar e$.

  Suppose first that the equalities~\eqref{eq:fji-equidivisible-2}
  hold. Consider the sequence of pseudowords %
  $\bigl(y(\bar uy)^{n!}\bigr)_n$, which converges %
  to~$y(\bar uy)^\omega=y(\bar uy)^{\omega-1}\cdot q_0y\cdot q_1$.
  Since the multiplication in~\Om CV is an open mapping and \Om CV~is
  compact, and taking again into account that \Om CV~is equidivisible,
  there are sequences of positive integers $(j_i)_i$, $(k_i)_i$, and
  $(\ell_i)_i$ such that %
  $\lim y(\bar uy)^{j_i}=y(\bar uy)^{\omega-1}$, %
  $\lim (\bar uy)^{k_i}=q_0y$, %
  $\lim(\bar uy)^{\ell_i}=q_1$, %
  and $(j_i+k_i+\ell_i)_i$ is a strictly increasing sequence of
  factorials. Since $y$~does not occur in~$q_0$, it follows that %
  $q_0=\bar u$. Since $\bar w\bar u=q_0$, by the preceding paragraph,
  we obtain $\bar w\bar u=\bar u$, which contradicts~\eqref{eq:fji-3}
  by Green's Lemma. The case where the
  equalities~\eqref{eq:fji-equidivisible-3} hold is handled similarly.
\end{proof}

Theorem~\ref{t:fji-equidivisible} applies in particular to
pseudovarieties of the form~$\bar{\pv H}$, with \pv H a nontrivial
pseudovariety of groups, but the conclusion is already part of
Corollary~\ref{c:fji-egs}. A new result is obtained by combining
Theorem~\ref{t:fji-equidivisible} with
Proposition~\ref{p:Cn-properties} and Corollary~\ref{c:wggm-Cn}, the
case $n=0$ being given by Corollary~\ref{c:fji-egs}.

\begin{Cor}
  \label{c:Cn-fji}
  For every pseudovariety of groups \pv H, the pseudovarieties %
  $\pv C_n\cap\bar{\pv H}$ are fji in the lattice %
  $\Cl L_0(\pv S)$.\qed
\end{Cor}

Corollary~\ref{c:Cn-fji} implies, in particular, that the complexity
pseudovarieties $\pv C_n$ are fji in the lattice $\Cl L(\pv S)$, which
solves the first part of \cite[Problem~43]{Rhodes&Steinberg:2009qt}.

We conclude with a connection between sfji and order primitivity. For
a pseudovariety \pv U of order semigroups, its \emph{order dual} is
the pseudovariety $\pv U^d$ consisting of the ordered semigroups
$(S,{\le})$ such that $(S,{\ge})$ belongs to~\pv U.

\begin{Lemma}[\cite{Pin&Weil:1994c}]
  \label{l:pw}
  Let \pv V be a pseudovariety of semigroups and let \pv U be a
  pseudovariety of ordered semigroups contained in~\pv V. Then \pv U
  generates \pv V as a pseudovariety of semigroups if and only if\/ %
  $\pv V=\pv U\vee\pv U^d$.
\end{Lemma}

The following is an immediate application of Lemma~\ref{l:pw}.

\begin{Prop}
  \label{p:sfji-vs-order-primitive}
  Every pseudovariety of semigroups which is sfji in the lattice %
  $\Cl L_o(\pv S)$ is order primitive.\qed
\end{Prop}

Combining Proposition~\ref{p:sfji-vs-order-primitive} with
Corollaries~\ref{c:fji-egs} and~\ref{c:Cn-fji}, we obtain the
following result.

\begin{Cor}
  \label{c:order-primitive-egs-bis}
  Let \pv H be a nontrivial pseudovariety of groups. Then the
  pseudovarieties \pv A, $\pv C_n\cap\bar{\pv H}$, $\bar{\pv H}$,
  $\pv{DS}\cap\bar{\pv H}$, and $\pv{CR}\cap\bar{\pv H}$ are order
  primitive.\qed
\end{Cor}

Interesting related pseudovarieties which do not fall in the realm of
application of Theorem~\ref{t:fji} are those of the form
$\pv{DO}\cap\bar{\pv H}$, in particular $\pv{DA}=\pv{DO}\cap\pv
A=\pv{DS}\cap\pv A$. We do not know whether \pv{DA} is fji or at least
sfji, within $\Cl L_o(\pv S)$ or even within $\Cl L(\pv S)$, or order
primitive. On the other hand, the pseudovariety $\pv
B=\pv{CR}\cap\bar{\pv H}$~is well known to be sfji in $\Cl L(\pv S)$,
a result that is an immediate consequence of the structure of the
lattice of varieties of~bands (see for instance
\cite[Section~5.5]{Almeida:1994a} for a diagram of the lattice and
bibliographic references). Taking into account that \pv B~is meet
distributive in $\Cl L(\pv S)$ \cite{Reilly&Zhang:1997}, it follows
that \pv B~is fji. By the discussion at the end of
Section~\ref{sec:orderability}, it follows that \pv B is also sfji
in~$\Cl L_o(\pv S)$.

\begin{table}[t]
  \centering
  \begin{tabular}{l||cc|cc|c|cc|}
    & & %\multicolumn{2}{c}{(almost)}
    & \multicolumn{2}{c|}{in $\Cl L(\pv S)$}
    & order- %
    & \multicolumn{2}{c|}{in $\Cl L_0(\pv S)$} \\
    \cline{4-5}\cline{7-8}
    pseudovariety
    & WGGM & GGM & sfji               & fji              & primitive %
     & sfji               & fji             \\
    \hline
    \pv A & Y & N & Y & Y & Y & Y & Y \\
    $\bar{\pv H}$ ($\pv I\ne\pv H\subseteq\pv G$)
    & Y & Y & Y & Y & Y & Y & Y \\
    $\pv C_n\cap\bar{\pv H}$ ($\pv I\ne\pv H\subseteq\pv G$)
    & Y & ? & Y & Y & Y & Y & Y \\
    $\pv{DS}\cap\bar{\pv H}$ ($\pv I\ne\pv H\subseteq\pv G$)
    & Y & ? & Y & Y & Y & Y & Y \\
    $\pv{DO}\cap\bar{\pv H}$ ($\pv H\subseteq\pv G$)
    & Y & N & ? & ? & ? & ? & ? \\
    $\pv{DS}\cap\pv C_n$ ($n\ge1$)
    & Y & ? & ? & ? & ? & ? & ? \\
    \pv B
    & almost & N & Y & Y & Y & Y & ? \\
    $\pv{CR}\cap\bar{\pv H}$ ($\pv I\ne\pv H\subseteq\pv G$)
    & Y & almost & Y & Y & Y & Y & Y \\
    $\pv{CR}\cap\pv C_n$ ($n\ge1$)
    & Y & ? & ? & ? & ? & ? & ? \\
    \multicolumn{7}{c}{}
  \end{tabular}
  \caption{Summary of results and open problems}
  \label{tab:summary}
\end{table}
Table~\ref{tab:summary} summarizes the results and questions about the
various pseudovarieties of concern in this paper. For each pair
pseudovariety, property, Y/N indicates whether or not the
pseudovariety enjoys the property, a question mark indicates that the
answer is presently unknown to the authors, and the word \emph{almost}
has the technical meaning introduced in Section~\ref{sec:actions}.

\subsection*{Acknowledgments}

The first author was partially supported by CMUP (UID/MAT/00144/2013),
which is funded by FCT (Portugal) with national (MEC) and European
structural funds through the programs FEDER, under the partnership
agreement PT2020. This work has been carried out partly during the
first author's sabbatical visit to the Department of Mathematics and
Statistics of Masaryk University, whose hospitality and support is
hereby gratefully acknowledged.

The second author was partially supported by the Grant 15-02862S of
the Grant Agency of the Czech Republic.

\providecommand{\bysame}{\leavevmode\hbox to3em{\hrulefill}\thinspace}
\providecommand{\MR}{\relax\ifhmode\unskip\space\fi MR }
\providecommand{\MRhref}[2]{%
  \href{http://www.ams.org/mathscinet-getitem?mr=#1}{#2}
}
\providecommand{\href}[2]{#2}

\end{document}